\documentclass[10pt]{article}
\usepackage{amsfonts, epsfig, amsmath, amssymb, color, amsthm}
\usepackage{mathrsfs}

\usepackage[round]{natbib}   

\usepackage[english]{babel}

\newcommand{\normal}{\color{black}}

\newcommand {\wt}[1]{ \widetilde{#1}}

\newcommand{\ba}{\begin{aligned}}
\newcommand{\ea}{\end{aligned}}

\newcommand{\T}{\mathrm{trace}}

\newcommand{\prt}{\partial}

\newcommand{\E}{\mathbf{E}}
\renewcommand{\P}{\mathbf{P}}

\renewcommand {\epsilon}{\varepsilon}

\theoremstyle{plain}
\newtheorem{thm}{Theorem}[section]
\newtheorem{lem}[thm]{Lemma}
\newtheorem{prp}[thm]{Proposition}
\newtheorem{cor}[thm]{Corollary}

\theoremstyle{definition}
\newtheorem{rem}[thm]{Remark}

\DeclareMathSymbol{\ophi}{\mathalpha}{letters}{"1E}

\newcommand{\e}{\varepsilon}
\newcommand{\eps}{\varepsilon}

\renewcommand{\phi}{\varphi}

\newcommand{\be}{\begin{equation}}
\newcommand{\ee}{\end{equation}}
\newcommand{\ben}{\begin{equation*}}
\newcommand{\een}{\end{equation*}}

\DeclareMathOperator{\sgn}{sgn}

\newcommand{\ex}{\mathrm{e}}
\newcommand{\di}{\mathrm{d}}

\newcommand{\rB}{\mathscr{B}}

\newcommand{\rF}{\mathscr{F}}

\newcommand{\bI}{\mathbb{I}}

\newcommand{\bR}{\mathbb{R}}

\allowdisplaybreaks[4]

\begin{document}


\title{Moment bounds for dissipative semimartingales with heavy jumps}



\author{Alexei Kulik\footnote{Wroclaw University of Science and Technology, Faculty of Pure and Applied Mathematics,
Wybrze\'ze Wyspia\'nskiego Str. 27, 50-370 Wroclaw, Poland; kulik.alex.m@gmail.com}\ \ and
Ilya Pavlyukevich\footnote{Institut f\"ur Mathematik, Friedrich--Schiller--Universit\"at Jena, Ernst--Abbe--Platz 2,
07743 Jena, Germany; ilya.pavlyukevich@uni-jena.de}}

\maketitle

\begin{abstract}
In this paper we show that if large jumps of an It\^o-semimartingale $X$ have a finite $p$-moment, $p>0$,
the radial part of its drift is dominated by $-|X|^\kappa$ for some $\kappa\geq -1$, and the balance condition $p+\kappa>1$ holds true, then
under some further natural technical assumptions one has
$\sup_{t\geq 0} \E |X_t|^{p_X}<\infty$ for each $p_X\in(0,p+\kappa-1)$. The upper bound $p+\kappa-1$ is generically optimal.
The proof is based on the extension of the method of Lyapunov functions to the semimartingale framework.
The uniform moment estimates obtained in this paper
are indispensable for the analysis of ergodic properties of L\'evy driven stochastic differential equations
and L\'evy driven  multi-scale systems.
\end{abstract}



\numberwithin{equation}{section}



\section{Introduction}


The goal of the present paper is to establish conditions on the uniform boundedness of the moments of a stochastic process $X$ over the infinite
time interval in a general semimartingale setting. We will assume that the process $X$ is an $n$-dimensional It\^o
semimartingale with a canonical decomposition
\be
\label{e:X}
X_t=X_0+A_t^{\leq 1}+ M_t+\int_0^t\int_{\bR^n}z\Big(N(\di z, \di s)-\bI_{|z|\leq 1 }\nu(\di z, \di s)\Big)
\ee
and predictable characteristics that are absolutely continuous with respect to the Lebesgue measure.
A typical example of a process $X$ that fits into this framework is a solution of an It\^o SDE
\begin{equation}
\label{e:eq1}
\begin{aligned}
\di X_t&=a(t,\omega,X_t)\,\di t+\sigma(t,\omega, X_t) \,\di W_t
+\int_{\mathbb{R}^n}c(t,\omega,X_{t-},z)\Big(Q(\di z,\di t)-\bI_{|z|\leq 1}\mu(\di z)\, \di t\Big)
\end{aligned}
\end{equation}
driven by a Brownian motion $W$ and a Poisson random measure $Q$ with a compensator $\mu(\di z)\,\di t$ and with
sufficiently regular coefficients $a$, $\sigma$ and $c$.
In this case
 \begin{equation}
\begin{aligned}
\di A_t^{\leq 1}&=a(t,\omega,X_t)\,\di t, \\
\di M_t&=\sigma(t,\omega, X_t)\,\di W_t,\\
N(A\times [0,t])&=\int_0^t\int_{\bR^n}\bI_{\{c(s,\omega,X_{s-},z)\in A\}}\,Q(\di z,\di s), \quad A\in \rB(\bR^n), \quad t\geq 0.
\end{aligned}
\end{equation}

We will study the long-time behaviour of $X$ under the following structural assumptions (for details, see Section \ref{s:assumptions} below):
\begin{itemize}
\item[(i)] the continuous martingale part and the ``small jumps'' part of $X$ are bounded in the sense of their characteristics;
\item[(ii)] the ``large jumps'' part of $X$ has a moment bound of some order $p>0$;
\item[(iii)]
the ``effective drift'' term $A$ (see \eqref{effective} below) performs ``contraction to the origin''. In other words, its
Radon--Nikodym density satisfies $\frac{\di A}{\di t}\cdot X\leq -\beta |X|^{1+\kappa}$ for some $\kappa\geq -1$
as long as $|X|$ is sufficiently large.
In certain sense this mimics 
the case of randomly perturbed gradient systems, as e.g.\ in \eqref{e:eq1} with $a(x)=-\nabla U(x)$ with a potential 
$|U(x)|\sim \beta|x|^{1+\kappa}$, $|x|\to\infty$.
We alert the reader that since the drift term $A^{\leq 1}$ in the canonical representation
depends on the cut-off function, an  extra care has to be taken for the estimates in the  case $\kappa<0$ which allows
the effective drift term to tend to $0$ at $\infty$; for more discussion see Remark \ref{r1} below.

\item[(iv)] the constants $p>0$ and $\kappa\geq -1$ satisfy the \emph{balance condition} $p+\kappa>1$.
\end{itemize}

In this generic setting, our main aim is to establish  bounds for the moments $\E|X_t|^{p_X}$ that are \emph{uniform in time}.
We will prove such bounds for the orders $p_X\in (0,p+\kappa-1)$ for $\kappa\not=1$ and $p_X\in(0,p]$ in the exceptional case $\kappa=1$.
An additional condition on the characteristics of the semimartingale will appear in the  critical case $\kappa=-1$.

The results we present here are strongly motivated by our ongoing research  of the \emph{stochastic averaging} effects in \emph{multi-scale systems} with jumps, where the moment bounds form a crucial component for an analysis of the limiting behaviour of such systems. 
For models with Brownian noise,
stochastic averaging has been studied systematically by e.g.\ \cite{PV1,PV2,PV3}.
Systems with jump noise are not yet well understood and exhibit
new effects which require a separate analysis. One such an effect is that the moments/tail behavior of the \emph{fast} 
component may cause substantially different limiting behavior of the \emph{slow} component. As an example, we mention a Langevin-type system with 
random jump perturbations of the  velocity component subject to a friction-type deceleration of the form $-|v|^\kappa\sgn v$, 
studied in \cite{EonGra-15} and \cite{KP-19}. It appears that, in the case where the jump distribution has a  heavy-tail of the order $\mathcal{O}(x^{-\alpha})$, the asymptotic behavior of the location component depends drastically on the value of $\alpha+2\kappa$.  When  $\alpha+2\kappa>4$ the location component requires a re-scaling to have a Gaussian weak limit (\cite{EonGra-15}), while in the case $\alpha+2\kappa<4$ the location component exhibits a 
non-Gaussian $\alpha/(2-\beta)$-stable weak limit (\cite{KP-19}). This dichotomy intuitively  
well corresponds to the one between the normal and stable domains of attraction in the central limit theorem, 
and supposedly should appear in various stochastic averaging problems with heavy-tailed jumps.

This potential field of applications motivates the general setting adopted in this paper and the questions studied. Namely,
for general multi-scale models with \emph{full coupling}, the {fast} component of the system is defined
by an SDE with the coefficients dependent of the \emph{slow} component. Then it clearly cannot be treated as an autonomous Markov process, 
which justifies the semimartingale setting we adopt. On the other hand, since the fast component operates 
at the `fast time scale' of the order $\eps^{-1}$,  the moment estimates for this component are
required for arbitrarily large $t\geq 0$.

In the Markovian setting, the moment bounds we are looking for are strongly related  to the ergodic properties of the process,
namely to the existence of the stationary probability measure, existence of moments of the stationary probability measure, and
estimation of ergodic rates, i.e.\ the rates of convergence of the marginal laws of the process to a stationary probability measure as $t\to \infty$. 
The questions of dissipativity, stability and ergodicity of Markov processes have been extensively studied, e.g.\ 
by \cite{kushner1967}, \cite{Hasminskii-80} (see also the second edition \cite{Khasminskii-12}) for diffusions and \cite{nummelin}, \cite{MT12} for general Markov chains.  Ergodic rates are naturally related to the moments of the passage times of a process to a ball centered at the origin.
The study of the passage times was first performed by \cite{lamperti1963criteria} for non-negative discrete time
Markov processes. Related results for non-negative discrete time adapted processes were obtained by
\cite{aspandiiarov1996passage,aspandiiarov1999general}. Continuous time processes
were studied in \cite{menshikov1996passage,menshikov2014explosion}.
Markov chains with heavy tail jumps were studied in \cite{belitsky2016random,georgiou2019markov}, see also a book by
\cite{menshikov2016non} for a self-contained exposition. Some of the results for the passage times were transferred to diffusions, see e.g.\
\cite[Theorem 3.1]{menshikov1996passage}. All these results mainly deal with the critical case $\kappa=-1$. The multivariate diffusion case was systematically analyzed by
\cite{veretennikov1997polynomial,veretennikov2000polynomial,veretennikov2001polynomial,malyshkin2001subexponential,klokov2004sub,uglov2017yet}
using the  argument  based on comparison with an one-dimensional diffusion. In the sequel we obtain moment estimates for the passage times  in a general semimartingale setting as a by-product of the moment estimates for the process itself.

The individual moment estimates are not yet completely studied even in the Markovian case. One can get an insight about the effects, which should appear while a dissipative drift is combined with heavy tails, from the results available about the moments/tails of the invariant probability measure for a Markov process defined by an SDE with L\'evy noise.  The explicit form for the stationary distribution for a SDE driven by a non-Gaussian L\'evy process is known only in a few particular cases.
In the \emph{linear} case, $a_t=-\beta X_t$, $\kappa=1$, the stationary law
exists if and only if $\int_{|z|\geq 1}\ln|z|\nu(\di z)<\infty$,
see \cite[Theorems 17.5 and 17.11]{Sato-99}, and its characteristic function can be calculated in the closed form. In particular,
one can conclude that for any $p>0$  a L\'evy noise with the tails of the order $\mathcal{O}(x^{-\alpha})$ yields a stationary distribution 
with the tails of the same order; this asymptotics can be seen as a manifestation of the linearity of the system. Considerable attention to the moments/tails of the stationary distribution of non-linear SDEs with additive L\'evy noise have been paid to at physical literature.
The stationary density was calculated explicitly for the  symmetric Cauchy process in a quartic potential well $U(x)=\beta\frac{x^4}{4}$
(i.e. $a_t=-\beta X_t^3$ and  $\kappa=3$, $p=1$), see \cite{ChechkinGKM-04}. An extension of this result to any  $\kappa=2m-1$, $m\in \{1,2,\dots\}$ was obtained by \cite{DubSpa07}. In these cases, the tails of the stationary measure
are of the order $\mathcal{O}(x^{-p-\kappa})$. It was shown by \cite{ChechkinGKM-05} on the physical level of rigour that
for $\kappa>0$  and $p\in (0,2)$ a stationary measure (if it exists) must have the tail of the order $\mathcal{O}(x^{-p-\kappa})$,
The sublinear case $\kappa<0$ is more subtle. However it was shown by \cite{DybSokChe10} by means of physical arguments that the  \emph{balance condition}
\be
\label{e:bc}
p+\kappa>1
\ee
should guarantee the existence of a stationary measure, at least for symmetric $p$-stable noises. For a thorough mathematical treatment of this topic we refer to \cite[Section~3.4]{Kulik17}, where it is shown in particular that under the balance condition \eqref{e:bc} the stationary law exists and has finite moment of any order $<p+\kappa-1$. All these results can be summarized as follows: a $\kappa$-dissipative drift transforms $p$-moment of the noise into (almost) $p+\kappa-1$-moment of the process.

In the non-stationary setting, the same moment transformation effect was observed in  case of superlinear drifts ($\kappa>1$) and symmetric additive L\'evy noise  by \cite{SamorodnitskyG-03}, see also an extension of this result to a
non-symmetric $p$-stable case by \cite{eongradinaru}.
\cite{kohatsu2003moments} performed an analysis of the moment properties in a finite time horizon of a closely related \emph{storage system}.
For sub-linear drifts ($\kappa<1$) the same effect can be seen for the time averaged moments; e.g. \cite[Theorem 4.1(i)]{DoucFortGui2009}, \cite[Section~2.8.2]{Kulik17}, or Proposition \ref{p1} (ii) below. The individual moment bounds in the sub-linear case $\kappa<1$ apparently  have not been  known even in the Markov case. The methods we develop in the current paper lead to  \emph{individual} and \emph{uniform in time}
moment bounds for the entire range of values $\kappa\in [-1,\infty)$ in a general multivariate semimartingale setting. We note that, in this wide generality, the moment transformation from $p$ (for the noise) to $p+\kappa-1$ (for the process) by a 
$\kappa$-dissipative drift is still observed; moreover, we will see in Section \ref{s:opt} 
that these moment bounds are actually non-improvable.

One of the principal methods for a study of stability of Markov processes is
the method of \emph{Lyapunov functions}, which consists in finding a (Lyapunov) function $V$
that satisfies the \emph{Lyapunov condition} $LV(x)\leq -\phi(V(x))+C$, where $L$ is the generator of the process and $\phi(v)$ is a 
certain function with $\phi(\infty)=\infty$. This method can be naturally extended to the semimartingale setting by changing the 
condition on the generator by a condition on the predictable part of the semimartingale decomposition for the process $V(X_t)$, see condition 
\textbf{L}$_{V, \gamma}$ below.  The Lyapunov condition is principally more general than the `drift dissipativity + noise tail bounds' 
assumptions we adopt. On the one hand, in Theorem \ref{tL} below we show that the Lyapunov condition follows from the dissipativity condition, 
and on the other hand, in certain situations, the Lyapunov condition may appear thanks to a ``stabilization by noise'' effect for 
a non-dissipative deterministic system, see e.g.\ \cite{kolba2012}. With this possibility in view, we present our results in the 
way which shows clearly which of them require the Lyapunov condition only, and for which a more detailed information about the 
semimartingale characteristics of the process is needed. We mention anyway that the drift dissipativity condition appears naturally in 
numerous physical models, e.g.\ the example in Section \ref{ex:1} below.

The paper is organized as follows. In Section \ref{s:assumptions} we formulate the problem. In Section \ref{sL},
the Lyapunov type condition in the semimartingale setting is formulated (Theorem \ref{tL}) and bounds for Ces\`{a}ro means of moments are derived (Proposition \ref{p1}).
In Section \ref{s:IM} the uniform moments bounds and results on the passage times are presented for the linear and sublinear drift $\kappa\in[-1,1]$
(Theorem \ref{t1}) and in Section \ref{s:superL} for the superlinear drift $\kappa>1$ (Theorem \ref{t2} and Theorem \ref{t3}).
In Section \ref{s:opt} we discuss the optimality of the balance condition.

Sections \ref{s21}, \ref{s:t1}, \ref{s:t2}, and \ref{s:t3}  are devoted to the proofs of  Theorem \ref{tL}, Theorem \ref{t1},
Theorem \ref{t2}, and Theorem \ref{t3} respectively.
Appendix \ref{a:1} contains several technical results used in the proof of Theorem \ref{t2}.
In Appendix \ref{a:2} we give an argument concerning the divergence of moments $\E|X_t|^{q_X}$ for $q_X>p+\kappa-1$ used in Section \ref{s:opt}.
The proof of Proposition \ref{p1} is  postponed to Appendix \ref{apC}.

\medskip

\noindent
\textbf{Acknowledgements.} This research was supported by the DFG project PA 2123/5-1
\emph{Asymptotic analysis
of multiscale L\'evy-driven stochastic Cucker--Smale and non-linear friction models}.
The work of A.\ Kulik was supported by the Polish National Science Center grant 2019/33/B/ST1/02923.
A.\ Kulik thanks FSU Jena for hospitality.
The authors thank the referees for their helpful comments and suggestions and careful reading of the manuscript.
 
\section{Setting and main results}
 
\noindent
\textbf{Notation.}
For $x,y\in\bR^n$, let $x\cdot y$ denote the scalar product in $\bR^n$, and let $|x|=\sqrt{x\cdot x}$ be the Euclidean norm in $\bR^n$.
The $n$-dimensional identity matrix is $\mathbf{I}_n$.
For a matrix
$B\in \bR^{n\times n}$, $|B|$ denotes its $2$-operator norm which is equal to the largest eigenvalue of $B$. The indicator function of a set $A$
is denoted by $\bI_A=\bI(A)$.

Throughout the paper we assume the filtered probability space
$(\Omega, \rF, (\rF_t)_{t\geq 0}, \P)$ to be fixed and satisfy the standard assumptions.

For a given stopping time $\sigma$,  we denote by
$\P_\sigma$ and $\E_\sigma$ the conditional probability and the conditional expectation w.r.t.\ $\rF_\sigma$.
Given a stopping time $\sigma$ and a level $R>0$, we denote
\be
\label{e:tauR}
\tau_R^\sigma=\inf\{t\geq \sigma\colon |X_t|\leq R\},
\ee
the first passage-time of the process $|X|$ under the level $R$ after $\sigma$. If $\sigma=0$, then we just omit the
upper subscript and write $\tau^0_R=\tau_R$. We also denote $\delta_R^\sigma=\tau_R^\sigma-\sigma$; note that $\delta_R^0=\tau_R$.

\subsection{Setting and assumptions\label{s:assumptions}}

Let $X=(X_t)_{t\geq 0}$ be an $\bR^n$-valued c\`adl\`ag It\^o semimartingale with the canonical representation (the L\'evy--It\^o decomposition)
\be
\label{eq1}
X_t=X_0+A_t^{\leq 1}+ M_t+\int_0^t\int_{\bR^n}z\Big(N(\di z, \di s)-\bI_{\{|z|\leq 1\}}\nu(\di z, \di s)\Big),
\ee
see \cite[Chapter II, \S 2c]{JacodS-03}. Here $A^{\leq 1}$ is a predictable process of locally finite variation, $A_0^{\leq 1}=0$, $M$ is a continuous local martingale
$M_0=0$, $N$ is the jump measure of $X$, and $\nu$ is its predictable compensator satisfying
\be
\label{e:nu_int}
\int_0^t\int_{\bR^n}(|z|^2\wedge 1) \nu(\di z, \di s)<\infty \quad \text{a.s.},\quad t\geq 0.
\ee
The subscript $^{\leq 1}$ for $A$ corresponds to the choice of the cut-off function $\bI_{\{|z|\leq 1\}}$. Another natural form of the canonical representation \eqref{eq1}
is
\be
\label{eqinfty}
X_t=X_0+A_t^{\infty}+ M_t+\int_0^t\int_{\bR^n}z\Big(N(\di z, \di s)-\nu(\di z, \di s) \Big),
\ee
provided that
\be
\label{e:finite}
\int_0^t\int_{|z|>1\normal }|z|\,\nu(\di z, \di s)<\infty \quad \text{a.s.},\quad t\geq 0.
\ee
In the latter case we denote
\be
A^\infty_t=A_t^{\leq 1}+\int_0^t\int_{|z|>1}z\,\nu(\di z, \di s),
\ee
and introduce the \emph{effective drift}
\begin{equation}
\label{effective}
A=\begin{cases}
     A^\infty,   &\text{if \eqref{e:finite} holds},\\
     A^{\leq 1}, & \text{otherwise},
    \end{cases}
\end{equation}
see discussion in Remark \ref{r1} below.

The It\^o semimartingale $X$ has semimartingale characteristics which are absolutely continuous with respect to the Lebesgue measure,
see \cite[Section 2.1.4]{jacod2011discretization}, i.e.\
\be
\begin{aligned}
A_t&=\int_0^t a_s\,\di s,\\
\langle M\rangle_t&=\int_0^t B_s\,\di s,\\
\nu(\di z,\di t)&=K_t(\di z)\,\di t
\end{aligned}
\ee
where $a = (a_t)$ is an $\bR^n$-valued process, $B = (B_t)$ is an $\bR^{n\times n}$-valued symmetric positive semi-definite process, and
$K_t = K_t (\omega, \di z)$ is a Radon measure on $\bR^n$ for each $(\omega, t)$.
The processes $t\mapsto a_t$, $t\mapsto b_t$ can be assumed to be progressively measurable, as well as $t\mapsto K_t (A)$ for all $A\in\rB(\bR^n)$.
Moreover \eqref{e:nu_int} implies that for all $t\geq 0$ the measures $K_t(\di z)$ satisfy
\be
\int_{\bR^n} (|z|^2 \wedge 1) K_t (\omega, \di z) < \infty\quad \text{a.s.}
\ee

We impose the following \emph{set of assumptions} \textbf{A} on the process \eqref{eq1}  that will be used throughout this paper.

\medskip

\noindent
\textbf{A$_M$} (the bound for the local martingale term): there is a constant $c_{\langle M\rangle}>0$ such that
the random matrices $(B_t)_{t\geq 0}$ satisfy
\be
|B_t|\leq c_{\langle M\rangle}\hbox{ a.s.,}\quad t\geq 0.
\ee
\begin{rem}
Under assumption \textbf{A}$_{M}$, the trace of $B_t$, $t\geq 0$, is uniformly bounded.
For the further reference needs, we introduce the minimal constant $c_{\T}>0$ such that for all $t\geq 0$
\be
\mathrm{Trace} (B_t)\leq c_{\T} \text{ a.s.}
\ee
Note that $c_{\T}\leq n c_{\langle M\rangle}$.
\end{rem}

We make the following assumptions about the small and large jumps of $X$.

\smallskip

\noindent
\textbf{A$_{\nu,\leq 1}$} (the small jumps condition): there exists $c_\nu>0$ such that for all $t\geq 0$
\be
\int_{|z|\leq 1} |z|^2\, K_t(\di z)\leq c_\nu\quad \text{a.s.}
\ee

\noindent
\textbf{A$_{\nu,p}$} (the large jumps condition): there exist $p>0$ and $c_{\nu,p}>0$ such that for all $t\geq 0$
\be
\int_{|z|>1} |z|^{p}\, K_t(\di z)\leq c_{\nu,p} \quad \text{a.s.}
\ee
\begin{rem}
\label{r:A2A3}
For $p\geq 2$, assumptions \textbf{A$_\nu$} and \textbf{A$_{\nu,p}$} imply that there is a
constant $c_{\langle N\rangle}>0$ such that for all $t\geq 0$
\be
\label{e:cN}
\int_{\bR^n} |z|^{2}\, K_t(\di z)\leq c_{\langle N\rangle}\quad \text{a.s.}.
\ee
\end{rem}
 
Next, we impose assumptions on the effective drift $A$ in terms of its density process $a$.

\noindent
\textbf{A}$_{a,\text{loc}}$ (the drift is locally bounded): for each $R>0$ there exists $C(R)>0$ such that for all $t\geq 0$
\be
|X_t|\leq R \quad  \Rightarrow \quad |a_t|\leq C(R) \quad \hbox{ a.s.}
\ee

\noindent
\textbf{A$_{a,\kappa}$} (the drift $\kappa$-contracts to the origin, dissipativity):
there exist $\kappa\geq -1$, $R_{0}>0$ and $\beta>0$ such that for all $t\geq 0$
\be
\label{drift}
|X_t|\geq  R_{0} \quad \Rightarrow \quad a_t\cdot X_t\leq -\beta|X_t|^{1+\kappa} \quad \hbox{ a.s.}
\ee

Finally, we impose the \emph{balance condition} between the `heavy tails' index $p$ and the `dissipativity' index $\kappa$
\noindent
\textbf{A}$_{\text{balance}}$ (the balance condition):
\be\label{balance}
p+\kappa>1.
\ee
 
\begin{rem}\label{r1}
  If $p\geq 1$, then \eqref{e:finite} holds and $a_t=\frac{\di}{\di t}A_t^\infty$. By \textbf{A$_{\nu,p}$}, the difference
\be
  \frac{\di}{\di t}A_t^\infty-\frac{\di}{\di t}A_t^{\leq 1}
\ee
is bounded, hence the assumption \textbf{A}$_{a,\text{loc}}$ is equivalent to the similar assumption formulated in the terms
of the original drift $A^{\leq 1}$. The same equivalence is true for the assumption \textbf{A$_{a,\kappa}$} if $\kappa>0$,
since  in this case $|x|^{1+\kappa}\gg |x|$, $x\to \infty$. That is, for $p\geq 1$ and $\kappa>0$ the particular choice of
the drift term is not essential, and one can verify the conditions
\textbf{A}$_{a,\text{loc}}$ and \textbf{A$_{a,\kappa}$} either for the original drift $A^{\leq 1}$ or for  $A^{\infty}$.

The difference becomes substantial either if $p<1$, i.e.\ when $A^\infty$ is not well defined,
or if $\kappa\leq 0$, when the assumption \textbf{A$_{a,\kappa}$} should be imposed on the ``fully compensated'' drift $A^\infty$.
Note that in the latter case the balance condition \eqref{e:bc} yields that $p>1$ and thus $A^\infty$ is well defined.
\end{rem}

\begin{rem}
In the critical case $\kappa=-1$, the balance condition \eqref{e:bc} implies that $p>2$ and Remark \ref{r:A2A3} applies.
\end{rem}

We also introduce a general \emph{Lyapunov condition} which will be systematically used in the paper. In what follows, a function
$V\colon\bR^n\to[1, \infty)$   such that $V(x)\to \infty$, $|x|\to \infty$ and a constant $\gamma>0$ are given.

\bigskip

\noindent
\textbf{L}$_{V, \gamma}$ (the Lyapunov condition): The process $V_t=V(X_t)$, $t\geq 0$, is a c\`adl\`ag semimartingale with the representation
\be
\label{DM}
\di V_t= a^V_t\,\di t+\di M^V_t,
\ee
where $M^V$ is a local supermartingale, $M^V_0=0$, and the drift $a^V$ satisfies the bound
\be \label{L}
a^V_t\leq  C_V-c_V V_t^\gamma
\ee
with some positive constants $C_V,c_V$.

The function $V$ from the above condition is commonly called the \emph{Lyapunov function} for $X$. It will be convenient for us to use the terminology from \cite{kolba2012} and call $V$ sub-, super-, and standard Lyapunov function for $\gamma\in (0,1), \gamma>1,$ and $\gamma=1$, respectively.

We note one important technical detail. In the literature devoted to stability of Markov processes, the Lyapunov condition appears frequently 
in the integral form, which actually requires $M^V$ to be a \emph{true} {super}-martingale, see, e.g., the drift assumption $\mathbf{D}(\mathbf{C},\mathbf{V}, \phi, \mathbf{b})$ in \cite{DoucFortGui2009}. The local supermartingale assumption in \textbf{L}$_{V, \gamma}$ is much easier to verify e.g.\ 
by using the It\^o formula. Furthermore at least for the simplest basic corollaries of 
the Lyapunov condition (see Proposition \ref{p1} below) its `local' version  is just as powerful as the `true' one thanks to the standard `time localization$+$Fatou's lemma trick', see e.g.\ the proof of Lemma 2.2 in
\cite{menshikov1996passage} and Section \ref{s41} below.

\subsection{Main results I: the Lyapunov condition and its immediate corollaries\label{sL}}

Our first main result establishes the Lyapunov condition for a semimartingale $X$ under the  assumptions {\textbf{A}}. The Lyapunov function $V$ will be any function  $V\in C^2(\bR^n,\bR)$  such that
\be
\label{e:V}
\begin{aligned}
V(x)&\geq 1,\quad x\in\bR^n,\\
V(x)&=|x|^{p},\quad |x|\geq 1,
\end{aligned}
\ee
where $p$ is taken from the assumption  \textbf{A}$_{\nu,p}$.
Without loss of generality we can and will also assume that $|V(x)-|x|^p|\leq 2$.

\begin{thm}
\label{tL} Let assumptions \emph{\textbf{A}} hold. In the critical case $\kappa=-1$, assume additionally that
\be
\label{cond_log}
c_{\T}+c_{\langle N\rangle}< 2\beta
\ee
 and  
\be
\label{cond_p}
p<2+\frac{2\beta-c_{\T}-c_{\langle N\rangle}}{c_{\langle M\rangle}+c_{\langle N\rangle}}.
\ee

 Then the Lyapunov condition \emph{\textbf{L}}$_{V, \gamma}$  holds true for any
$V\in C^2(\bR^n,\bR)$ satisfying \eqref{e:V} and
\be\label{gamma}
\gamma=\frac{p+\kappa-1}{p}.
\ee
\end{thm}

Note that, by Theorem \ref{tL},  the sub-linear, linear, or super-linear drift yields that the function $|x|^p$, mollified at
the origin, is a sub-, standard or super-Lyapunov function, respectively.
The proof of Theorem~\ref{tL}
is given in Section \ref{s21} below. Here we give three simple straightforward corollaries of the Lyapunov condition \textbf{L}$_{V, \gamma}$.

\begin{prp}\label{p1} Let the Lyapunov condition \emph{\textbf{L}}$_{V, \gamma}$ hold true. Then for each stopping time $\sigma$
\begin{itemize}
\item[(i)]
\be\label{moment_finite}
\E_\sigma V_t\leq  V_\sigma+C_V(t-\sigma)\quad \text{on}\quad  \{t\geq \sigma\};
\ee
\item[(ii)]
\be
 \frac{1}{t-\sigma}\E_\sigma\int_\sigma^{t} V_s^\gamma\, \di s\leq  \frac{V_\sigma}{c_V (t-\sigma)} +\frac{C_V}{c_V}
\quad \text{on}\quad  \{t> \sigma\};
\ee
\item[(iii)] if $\gamma\geq 1$, then
\be
\E_\sigma V_t\leq \Upsilon(t-\sigma, V_\sigma)
\quad \text{on}\quad  \{t\geq \sigma\},
\ee
where $\Upsilon(t,\upsilon)$ denotes the solution to the Cauchy problem for the ODE
$$
\Upsilon'_t=C_V-c_V\Upsilon_t^\gamma, \quad \Upsilon_0=\upsilon.
$$

\end{itemize}
\end{prp}

In the case of $M^V$ in \eqref{DM} being a true supermartingale, statements (i), (ii) can be  obtained by direct integration
of this inequality, while (iii) follows by the Jensen inequality and comparison theorem for ODEs, see which then coincides with
\cite[Lemma 6.1]{kolba2012}. With our local supermartingale based definition of the Lyapunov function,
the proofs can be made essentially in the same lines with the help of the ``time localization$+$Fatou's lemma trick'', see Appendix \ref{apC} below.

Since the difference $V(x)-|x|^p$ is bounded by 2, the next Corollary is straightforward.
\begin{cor}\label{c02}
Let  assumptions \emph{\textbf{A}} hold true. Then there exists a constant $C>0$ such that for any stopping time $\sigma$
\begin{itemize}
\item[(i)]
\be
\label{b1}
\E_\sigma |X_t|^p\leq  |X_\sigma|^p+C (t-\sigma)+C\quad \text{on}\quad  \{t\geq \sigma\};
\ee
\item[(ii)]
\be
\label{b2}
 \frac{1}{t-\sigma}\E_\sigma\int_\sigma^t|X_s|^{p+\kappa-1}\, \di s\leq  \frac{C}{t-\sigma }|X_\sigma|^p+\frac{C}{t-\sigma}+C
\quad \text{on}\quad  \{t> \sigma\}.
\ee
\item[(iii)] if $\gamma>1$, then
\be
\E_\sigma |X_t|^p\leq \Upsilon(t-\sigma, |X_\sigma|^p+1)+1
\quad \text{on}\quad  \{t\geq \sigma\},
\ee
where $\Upsilon(t,\upsilon)$ denotes the solution to the Cauchy problem for the ODE
$$
\Upsilon'_t=C_V-c_V\Upsilon_t^\gamma, \quad \Upsilon_0=\upsilon.
$$
\end{itemize}
\end{cor}
\normal

Let us summarize. The Lyapunov type condition \eqref{L}  yields directly a time-dependent bound \eqref{b1} for the $p$-th moment of $X$,
and a time-independent (for $t$ separated from $\sigma$) bound \eqref{b2} for Ces\`{a}ro means of
the moments of the order $p+\kappa-1$.
The latter estimate corresponds well to the long-term behaviour of the process $X$.
In particular in the Markovian setting, the estimate \eqref{b2} is naturally related to the moment bounds for the invariant measure of the process,
e.g.~\cite[Section~2.8]{Kulik17}.
Still, for various applications it would be useful to have
\emph{time-independent} and \emph{individual} moment bounds. Such a bound, in the standard- and super-Lyapunov cases $\gamma=1$ and $\gamma>1$,  is provided by statement (iii), which however does not reveal the effect of increasing the order of the moment from $p$ to $p+\kappa-1$ for super-linear drifts.

In what follows, we focus on this more delicate type of estimates, namely, on individual and time-independent moments estimates
of the order close to  $p+\kappa-1$. In Section \ref{s:opt} below we will see that this value is actually optimal.

\subsection{Main results II: individual moment estimates for sub-linear and linear drifts\label{s:IM}}

Our second main theorem provides uniform in  $t\geq 0$ moment bounds in  the cases
of sub-linear and linear bounds on the drift. As a by-product of the proof we also obtain
the passage-times moment estimates. The latter estimates are not essentially new and have numerous analogues in the literature, e.g.\
\cite[Theorem 2.1]{menshikov1996passage} (critical $\kappa=-1$, continuous semimartingale $X$), or
\cite[Theorem 4.1(ii)]{DoucFortGui2009}
(general sub-linear drift, Markovian case).
We provide them here in order to make it easy for the reader to see the entire picture.
We recall the notation $\delta_\sigma^R=\tau_\sigma^R-\sigma$ for the time spent by $X$ after a stopping time $\sigma$
before the passage of the process $|X|$ under the level $R$.

\begin{thm}
\label{t1}
\begin{itemize}
\item[(i)] Let the Lyapunov condition \emph{\textbf{L}}$_{V, \gamma}$ hold true with $\gamma<1$. Then for any $\gamma'<\gamma$ we have
\be
\label{moment_sub-V}
\E_\sigma V_t^{\gamma'}\leq C+V_\sigma^{\gamma'} \quad \text{on the event}\quad  \{t\geq \sigma\}.
\ee
For $\gamma=1$, the same statement holds true with $\gamma'=\gamma$.
\item[(ii)]
Let assumptions \emph{\textbf{A}} hold true with  $\kappa\in [-1,1)$.  For $\kappa=-1$ assume in addition \eqref{cond_log} and \eqref{cond_p}.
Then for any $p_X\in (0, p+\kappa-1)$  there exists a constant $C>0$ such that for any stopping time $\sigma$
\be\label{moment_sub}
\E_\sigma |X_t|^{p_X}\leq C+|X_\sigma|^{p_X} \quad \text{on the event}\quad  \{t\geq \sigma\}.
\ee
In addition, for $R>0$ large enough
\be\label{PT_sub}
\E_\sigma (\delta_R^\sigma)^\frac{p}{1-\kappa}\leq |X_\sigma|^p.
\ee
For $\kappa=1$, inequality \eqref{moment_sub} holds true also for $p_X=p$. Moreover for any $c<p\beta$ there exists $C>0$
such that for $R>0$ large enough
\be
\label{PT_lin}
\E_\sigma \ex^{c\delta_R^\sigma}\leq C+|X_\sigma|^p \quad \text{on the event}\quad  \{t\geq \sigma\}.
\ee
            \end{itemize}
\end{thm}

\begin{rem} Clearly, the estimate \eqref{moment_sub} is a straightforward corollary from \eqref{moment_sub-V} and Theorem \ref{tL}.
\end{rem}

\begin{rem}
Estimate \eqref{PT_sub} was obtained in Theorem 4 by \cite{veretennikov1997polynomial} for Markovian diffusions with $\kappa=-1$
under the assumption
\be
\label{cond_Ver}
p<2r_0-1, \quad r_0=\frac{1}{c_{\langle M\rangle}}\Big(\beta-\frac{c_{\T}-\lambda_-}{2}\Big),
\ee
where $\lambda_-$ is the uniform lower bound for the smallest eigenvalue of $B_t$, $t\geq 0$. That is,
in Theorem 4 in \cite{veretennikov1997polynomial} it is actually assumed that
\be
\label{cond_Ver2}
p<\frac{1}{c_{\langle M\rangle}}\Big(2\beta-c_{\T}+\lambda_-\Big)-1=\frac{2\beta-c_{\T}}{c_{\langle M\rangle}}-\frac{c_{\langle M\rangle}-\lambda_-}{c_{\langle M\rangle}}.
\ee
On the other hand, in the continuous case $c_{\langle N\rangle}=0$, and thus \eqref{cond_p} has the form
\be
\label{cond}
p<2+\frac{2\beta-c_{\T}}{c_{\langle M\rangle}}.
\ee
Since $\lambda_-\leq c_{\langle M\rangle}$, this condition is obviously weaker than \eqref{cond_Ver2}.
Calculation in Section \ref{s:opt} below shows that  condition \eqref{cond} is eventually optimal.
\end{rem}

\subsection{Main results III: individual moment estimates for super-linear drifts\label{s:superL}}

Our last main theorem deals with the case of super-linear drift. The crucial difference to the previous case is that the corresponding
moment and passage time bounds can be made uniform with respect to the initial value of the process.
This agrees well with the intuition that the behavior
of $|X|$ should be qualitatively comparable to that of the solution to the ODE
\be
\label{ODE}
\di x_t=-\beta |x_t|^{\kappa}\sgn x_t\,\di t,
\ee
which,  for $\kappa>1$,  ``returns from the infinity'' to a bounded region in finite time.

\begin{thm}\label{t2}  Let assumptions \emph{\textbf{A}} hold with $\kappa>1$.
Then for any $t_0>0$ and  $p_X\in (0, p+\kappa-1)$  there exists a constant $C=C(t_0, p_X)>0$ such that
\be
\label{moment_super}
\E_\sigma |X_t|^{p_X}\leq C \quad \hbox{on the event} \quad \{t\geq \sigma+t_0\}.
\ee
In addition, for any  $q>0$ there exists a constant $C$ such that for $R>0$
\be\label{PT_super}
\P_\sigma\Big(\delta^\sigma_R> \frac{C}{R^{\kappa-1}}\Big)\leq \frac{C}{R^{q}}.
\ee
\end{thm}

 The estimate \eqref{PT_super} actually tells us  that the passage-time below the level $R$ for the process $|X|$ is comparable
 with $R^{1-\kappa}$, which is essentially the first passage-time for the solution
 \be\label{sol_ODE}
 x_t=\big(\beta(\kappa-1)t\big)^{-\frac{1}{\kappa-1}},\quad t>0,
 \ee
  to ODE \eqref{ODE}, starting ``at infinity''.

  The moment bound \eqref{moment_super} has a certain drawback: the constant $C=C(t_0, p_X)$ tends to $\infty$ when $t_0\to 0+$. This well corresponds to the fact that the ODE \eqref{ODE} starting ``at infinity'' needs a certain positive time  to reach a fixed level. On the other hand, in some cases it might be useful to have a version of \eqref{moment_super} valid for all $t_0>0$; in particular, it is visible that such a version will be needed  in a study of L\'evy-driven multi-scale systems.
  By analogy with the small time behavior of the ODE \eqref{ODE}, one can see clearly that such an estimate  should involve the initial value of the process. We give one such a version in the following theorem.

\begin{thm}\label{t3} Let assumptions \emph{\textbf{A}} hold true with $\kappa>1$.
Then for any $p_X\in (0, p+\kappa-1)$ there exists a constant $C>0$ such that for any stopping time $\sigma$
\be
\label{moment_super1}
\E_\sigma |X_t|^{p_X}\leq C +|X_\sigma|^{p_X}\quad \hbox{on the event} \quad \{t\geq \sigma\}.
\ee
\end{thm}
\begin{rem} The estimate \eqref{moment_super1} has exactly the same form with \eqref{moment_sub} obtained in the case $\kappa<1$. That is, combining these two cases we have that \eqref{moment_super1} holds true under the assumptions {\textbf{A}} for any $\kappa\geq -1$ (in the critical  case $\kappa=-1$
extra bounds  \eqref{cond_log}, \eqref{cond_p} for the constants should be involved).
Though, the proofs of  \eqref{moment_super1} in the cases $\kappa\leq 1$ and $\kappa>1$ are substantially different.
\end{rem}

\subsection{Example: heavy tail perturbations of the Lorenz-84 model\label{ex:1}}

We illustrate the results by an example of a randomly perturbed Lorenz-84 model with modified dissipativity, see \cite{lorenz1984irregularity}.
This model was
defined by its author as the `simplest possible model' capable of representing general
atmosphere circulation.
Let $a>0$, $b\in\bR$, $\mathbf{L}=(L^1,L^2,L^3)$ is a three-dimensional $\alpha$-stable L\'evy process, $\alpha\in (0,2)$,
and let $\mathbf{X}:=(X,Y,Z)$ be a solution to the
three-dimensional SDE
\be
\label{e:L84}
\begin{aligned}
\di X_t&= (-Y^2_t-Z^2_t -a X_t)\phi(\mathbf{X}_t)\,\di t +\psi_1(\mathbf{X})_{t-}\,\di L^1_t \\
\di Y_t&= (X_tY_t -b X_tZ_t-Y_t)\phi(\mathbf{X}_t)\,\di t+ \psi_2(\mathbf{X})_{t-}\,\di L^2_t \\
\di Z_t&= (bX_tY_t+X_tZ_t-Z_t)\phi(\mathbf{X}_t)\,\di t+ \psi_3(\mathbf{X})_{t-}\,\di L^3_t ,
\end{aligned}
\ee
where the functions
$\psi_i$, $i=1,2,3$, are bounded and functional Lipschitz
and $\phi=\phi(\mathbf X_t)$ is random locally Lipschitz and such that for some $\gamma\in\bR$, $c>0$ and $R_0>0$
\be
\phi(\mathbf X_t)= c|\mathbf X_t|^\gamma,\quad |\mathbf X_t|\geq R_0,
\ee
see \cite[Chapter V.3]{Protter-04} for the definitions of the Lipschitz properties in the stochastic semimartingale setting.
For $\phi\equiv 1$ and $\psi\equiv 0$
we obtain the classical deterministic Lorenz-84 model. The function $\phi$ increases the dissipativity of the system for $\gamma>0$
 and reduces its dissipativity for $\gamma<0$.

Assumptions \textbf{A}$_M$, \textbf{A}$_{\nu,\leq 1}$ and \textbf{A}$_{a,\text{loc}}$ are trivially satisfied.
Assumption \textbf{A}$_{\nu,p}$ holds for each $p\in(0,\alpha)$.
To check Assumption \textbf{A}$_{a,\kappa}$
we estimate the drift term as
\be
\begin{aligned}
\phi(\mathbf{X}_t)\Big[(-Y_t^2-Z_t^2 -a X_t)X_t&+(X_tY_t -b X_tZ_t-Y_t)Y_t+(bX_tY_t+X_tZ_t-Z_t)Z_t\Big]\\
&=-\phi(\mathbf{X}_t)\Big[aX^2_t+Y^2_t+Z^2_t]
\leq -c (a\wedge 1)|\mathbf{X}_t|^{2+\gamma},\quad |\mathbf X_t|\geq R_0.
\end{aligned}
\ee
Therefore Assumption \textbf{A}$_{a,\kappa}$ holds with $\kappa=\gamma+1$ and $\beta=c(a\wedge 1)$.
Then by Theorems \ref{t1} and \ref{t3}, for each $\gamma>-\alpha$ and
$p_X\in (0,\alpha+\gamma)$ there is $C>0$ such that for any initial point $\mathbf{X}_0\in\bR^3$
\be
\sup_{t\geq 0}\E |\mathbf{X}_t|^{p_X}\leq C+ |\mathbf{X}_0|^{p_X}.
\ee

\subsection{Optimality of the balance condition and the bounds \eqref{cond_log} and \eqref{cond_p}\label{s:opt}}

We emphasize that the uniform-in-time moment bounds obtained in Theorem \ref{t1} and Theorem \ref{t2} are close to optimal in the sense that
 the balance condition \eqref{e:bc} determines a generic upper bound for the moments of $X$ to exist or to be bounded
over an infinite period of time. We demonstrate that the moments of the order $q_X>p+\kappa-1$ are infinite or unbounded on the example
of the so-called storage system considered by \cite{kohatsu2003moments}.

Let
$X$ be a solution of the one-dimensional SDE
\be
\label{e:storage}
\di X_t=-r(X_t)\, \di t+\int_{z>1} z N(\di z,\di t)   ,\quad X_0\geq 0,
\ee
with $r\colon \bR_+\to\bR_+$ being locally Lipschitz continuous and $r(x)=x^\kappa$, $x>1$, $\kappa\geq -1$, and
$N$ being a Poisson random measure with intensity $\nu$ such that for some $\alpha>0$
\be
\int_{z>x}\nu(\di z)=\frac{1}{x^\alpha}, \quad x>1.
\ee
There is no continuous martingale part and no small jumps.
Clearly, the process $X$ is non-negative.
In such a setting, assumptions \textbf{A} are satisfied with any $0<p<\alpha$.

Now, let us fix $q_X>\alpha+\kappa-1$ and
show that a moment of the order $q_X$ does not follow the bounds from Theorems \ref{t1} and \ref{t2}.

Indeed, for $\kappa>1$, $\E X^{q_X}_t=+\infty$ for each $t>0$ by
Example 3.2 from \cite{kohatsu2003moments} or by Theorem 3.1 from \cite{SamorodnitskyG-03} (in the latter case one has to consider a symmetric
L\'evy process $Z$ and to extend the drift to the negative half-line).

For $\kappa<1$, it is shown in Example 4.2 by \cite{kohatsu2003moments} that
$\E X_t^{q_X}<\infty$, $t>0$ whenever $q<\alpha$.
However for $q_X\geq \alpha+\kappa-1$
\be
\liminf_{t\to\infty}\E X^{q_X}_t=+\infty.
\ee
This is demonstrated in Appendix \ref{a:2}.

Next, for $\kappa=-1$ we consider the diffusion equation $\di X=-\nabla U(X)\,\di t+\sigma \di W$ with $U(x)=\frac{\beta}{2}\ln(1+x^2)$, $\beta>0$, and 
the stationary density 
\be
\label{e:mu}
\rho(x)=c\,\ex^{-2 U(x)/\sigma^2} =\frac{c}{(1+x^2)^{\beta/\sigma^2}}.
\ee
$c>0$ being a normalizing constant.
A straightforward calculation shows that all $p_X$-moments of the stationary measure are finite for
\be
\label{e:pX}
p_X<\frac{2\beta}{\sigma^2}-1
\ee
and $q_X$-moments are infnite for $q_X\geq \frac{2\beta}{\sigma^2}-1$.
In this case, $c_\T=c_{\langle M\rangle}=\sigma^2$ and $c_{\langle N\rangle}=0$.
The condition \eqref{cond_log} takes the form $\sigma^2<2\beta$ and coincides with the condition
for the existence of the stationary distribution. Condition \eqref{cond_p} takes the form
\be
\label{e:pp}
p<1+\frac{2\beta}{\sigma^2}.
\ee
Hence each $p_X$-moment is finite for $p_X\in (0,p+\kappa-1)$ for some $p$ satisfying \eqref{e:pp} if and only if \eqref{e:pX} holds.
This shows the generic optimality of our conditions \eqref{cond_log} and \eqref{cond_p} in the diffusion case.

\section{Proof of Theorem \ref{tL}}\label{s21}
Our proof of Theorem \ref{tL} follows the strategy which was used in \cite[Section~3.4]{Kulik17} in order to verify the Lyapunov condition in the Markovian setting. We will apply the It\^o formula to $V(X_t)$ and analyze different parts of the drift $a^V$ which arise from the different parts
of the semimartingale representation for $X$, namely, form the drift, the continuous martingale, the small jumps, and the large jumps
respectively. Such an analysis will be made in slightly different ways in the
 three cases $p\in (0,1), p\in [1,2],$ and $p>2$, which we thus consider separately.

\noindent
\textbf{Case I: $p\in(0,1)$.} According to \eqref{effective},
$a_t=\frac{\di}{\di t}A_t^{\leq 1}$, and we use the canonical decomposition \eqref{eq1}. By the It\^o formula,
\be
\label{Ito1}
\ba
V(X_t)&=V(X_0)+\int_0^t\nabla V(X_s)\cdot a_s\, \di s+\frac12\int_{0}^{t}\nabla^2V(X_s)\cdot \di\langle M\rangle_s+\int_0^t\nabla V(X_{s})\cdot \di M_s
\\&+\int_0^t\int_{|z|\leq 1}\Big[V(X_{s-}+z)-V(X_{s-})\Big]\,\Big(N(\di z, \di s)-\nu(\di z, \di s)\Big)
\\&+\int_0^t\int_{|z|\leq 1}\Big[V(X_{s-}+z)-V(X_{s-})-\nabla V(X_{s-})\cdot z\Big]\,\nu(\di z, \di s)
\\&+\int_0^t\int_{|z|> 1}\Big[V(X_{s-}+z)-V(X_{s-})\Big]\,N(\di z, \di s).
\ea
\ee
For any $R>0$, the functions $V$, $\nabla V$, $\nabla^2 V$ are bounded on the set $\{|x|\leq R\}$ and there exists $C =C_R>0$ such that
\be
V(x+z)\leq C(1+|z|^p), \quad |x|\leq R.
\ee
Hence \eqref{Ito1} yields the semimartingale representation \eqref{DM} with the local martingale
\be
M^V_t=\int_0^t\nabla V(X_{s})\cdot \di M_s+ \int_0^t\int\Big[V(X_{s-}+z)-V(X_{s-})\Big]\,\Big(N(\di z, \di s)-\nu(\di z, \di s)\Big).
\ee
We can even choose explicitly the localization sequence for $M^V$, namely we can take the stopping times
\be
\theta_k=\inf\{t\geq 0\colon |X_t|\geq k\}, \quad k\geq 1.
\ee
The drift term is constituted by four summands,
\be
\label{a}
a_t^V = a_t^D+a_t^{M}+a_t^{N}+a_t^J,
\ee
where
\begin{align}
\label{e:aD}
a_t^D &=  \nabla V(X_{t})\cdot  a_t, \\
a_t^M &=\frac12   \operatorname{Trace} \Big( \nabla^2V(X_{t})\cdot B_t\Big),
\end{align}
and $a^{N}$ and $a^J$ are defined by the identities
\begin{align}
\label{e:aN}
a_t^{N} &= \int_{|z|\leq 1}\Big(V(X_{t}+z)-V(X_{t})-\nabla V(X_{t})\cdot z\Big)\,K_t(\di z),\\
a_t^{J} &= \int_{|z|>1}\Big(V(X_{t}+z)-V(X_{t})\Big)\, K_t(\di z).
\end{align}
We analyze separately the terms in the decomposition \eqref{a}.

\noindent
\textbf{1. Term $a^D$.} Without loss of generality, we can and will assume that $R_0$ in the assumption \textbf{A$_{a,\kappa}$} satisfies $R_0>1$. Then by assumption \textbf{A$_{a,\kappa}$}
\be
\nabla V(X_t)\cdot  a_t=p|X_t|^{p-2}X_t\cdot a_t\leq - p\beta|X_t|^{p-2}|X_t|^{1+\kappa}=  -p\beta |X_t|^{p+\kappa-1}, \quad |X_t|\geq R_0.
\ee
On the other hand, for $|X_t|\leq R_0$ the term $a_t^D$ is bounded by assumption \textbf{A}$_{a,\text{loc}}$.
Hence  there exists $C_1>0$ such that
\be
\label{bound_drift}
a^D_t\leq C_1-p\beta|X_t|^{p+\kappa-1}.
\ee

\noindent
\textbf{2. Term $a^M$.}
We have
\be
\label{e:nablaV}
\nabla^2V(x)=p|x|^{p-2}\mathbf{I}_{n}+p(p-2)|x|^{p-4}(x\otimes x), \quad |x|> 1.
\ee
Therefore since $p\leq 1<2$ and $B_t\geq 0$, for $|X_t|>1$ one has
\be
\label{e:aM}\ba
a^M_t& =\frac12p|X_t|^{p-2}\mathrm{Trace}\, B_t+\frac12p(p-2)|X_t|^{p-4}(B_t X_t)\cdot X_t
\\&\leq \frac{p}{2}\Big(c_{\T}+(p-2)_+c_{\langle M\rangle}\Big)|X_t|^{p-2}
\\& =\frac{p}{2}c_{\T}|X_t|^{p-2}.
\\&\leq \frac{p}{2}c_{\T}.
\ea\ee
Since $\nabla^2 V(x)$ is bounded on $|x|\leq 1$ and $B_t$ is bounded by assumption \textbf{A$_M$}, this yields that for any $X_t$
\be\label{bound_diff}
a^M_t\leq C_2
\ee
for some $C_2>0$.

\noindent
\textbf{3. Term $a^N$.}
Since
\be\label{Taylor}
V(x+z)-V(x)-\nabla V(x)\cdot z =\frac{1}{2}z^T\cdot \nabla^2V(x+q z)\cdot z
\ee
with some $q=q(x,z)\in (0,1)$, for any $\gamma>1$ we can fix $R=R(\gamma)$ such that for $|x|> R$ and $|z|\leq 1$,
\be
\label{e:VTaylor}
\begin{aligned}
V(x+z)-V(x)-\nabla V(x)\cdot z
&=\frac{p}{2}\Big(|x+q z|^{p-2}|z|^2+(p-2)|x+q z|^{p-4}(x\cdot z+q |z|^2)^2\Big)\\
&\leq \gamma\frac{p}{2}\Big((p-2)_++1\Big)|x|^{p-2}|z|^2
\\&=\gamma\frac{p}{2}|x|^{p-2}|z|^2
\\&\leq \gamma\frac{p}{2}|z|^2.
\end{aligned}
\ee
For $|x|\leq R=R(\gamma)$ and $|z|\leq 1$, there is $C_3=C_3(\gamma)>0$ such that
\be
|V(x+z)-V(x)-\nabla V(x)\cdot z|\leq C_3|z|^2.
\ee
This yields
\be
\label{bound_small}
\begin{aligned}
a^N_t & = a^N_t\bI(|X_t|>R) + a^N_t\bI(|X_t|\leq R)\\
&\leq \gamma\frac{p}{2} \int_{|z|\leq 1} |z|^2\, K_r(\di z)+ C_3 \int_{|z|\leq 1} |z|^2 K_r(\di z)
\leq C_4.\\
\end{aligned}
\ee

\noindent
\textbf{4. Term $a^J$.}
To analyze the last term $a^J$ we recall the simple inequality valid for $p\in (0,1]$:
\be
|x+z|^p-|x|^p\leq |z|^p,\quad x\in\bR^n,\quad  x,z\in\bR^n.
\ee
Since $| V(x)-|x|^p|\leq 1$, this yields by assumption \textbf{A$_{\nu,p}$} that
\be
a_t^J  \leq  \int_{|z|>1}(2+|z|^p)K_t(\di z) \leq C_5. 
\ee
Summarizing the above estimates we get that for some $C>0$
\be
\label{e:aV-1}
a^V_t\leq C -p\beta |X_t|^{p+\kappa-1}.
\ee
Since $p+\kappa>1$ and $V(x)\leq |x|^{p}+1$, this yields that for any $c_V<p\beta$ there exists $C_V$ large enough such that \eqref{L} holds.

\smallskip

\noindent
\textbf{Case II: $p\in [1,2]$.} Now $a_t=\frac{\di}{\di t}A_t^{\infty}$ and we use the canonical decomposition \eqref{eqinfty}. By the It\^o formula,
\be\label{Ito}
\ba
V(X_t)&=V(X_0)+\int_0^t\nabla V(X_s)\cdot a_s\, \di s+\frac12\int_{0}^{t}\nabla^2V(X_s)\cdot \di\langle M\rangle_s+\int_0^t\nabla V(X_{s-})\cdot \di M_s
\\&+\int_0^t\int\Big[V(X_{s-}+z)-V(X_{s-})\Big]\,\Big(N(\di z, \di s)-\nu(\di z, \di s)\Big)
\\&+\int_0^t\int\Big[V(X_{s-}+z)-V(X_{s-})-\nabla V(X_{s-})\cdot z\Big]\,\nu(\di z, \di s).
\ea
\ee
The decomposition \eqref{a} of the drift holds true,  and the summands
$a^D$, $a^M$, $a^N$ have the same form as in Case I and follow literally the same estimates.  The term $a^J$ now has the form
\be
\label{e:aJ2}
a^J_t= \int_{|z|>1}\Big[V(X_t+z)-V(X_t)-\nabla V(X_t)\cdot z\Big]\,K_t(\di z).
\ee
 To estimate this term, we fix $\eps\in (0,1)$ and decompose
\be\label{aJ}
\begin{aligned}
a_t^{J}
&=  \int_{1<|z|\leq \eps|X_t|}\Big[V(X_t+z)-V(X_t)-\nabla V(X_t)\cdot z\Big]\,K_t(\di z)\\
&+ \int_{|z|>\eps|X_t|}\Big[V(X_t+z)-V(X_t)-\nabla V(X_t)\cdot z\Big]\,K_t(\di z)
=:I_{t}^1+I_{t}^2.
\end{aligned}
\ee
For $|z|>\eps|x|$ we have
\be
V(x+z)-V(x)-\nabla V(x)\cdot z\leq V(x+z)+|\nabla V(x)||z|\leq C_6(|x|^p+|z|^p+|x|^{p-1}|z|+1)\leq C_7(|z|^p+1)
\ee
with some $C_6$, $C_7=C_7(\eps)>0$, which by  assumption \textbf{A$_{\nu,p}$}  yields the bound
\be
\label{bound_large}
I^2_t\leq C_8,\quad C_8=C_8(\eps)>0.
\ee
To estimate $I^1_t$, we note that
for $|z|\leq \eps |x|$ and $|x|>(1-\eps)^{-1}>1$ we have $|x+z|\geq |x|-|z|>1$ and thus
by the Taylor formula, \eqref{e:nablaV} and \eqref{Taylor}
\be
\label{bound_p<2}
\begin{aligned}
V(x+z)-V(x)-\nabla V(x)\cdot z&=|x+z|^p-|x|^p-p|x|^{p-2}x\cdot z\\
&\leq \frac{p }{2}|x+qz|^{p-2}|z|^2
\leq \frac{p }{2}|x|^{p-2}\Big|1+q\frac{z}{|x|}\Big|^{p-2}|z|^2
\\&
\leq C_4|z|^p,
\end{aligned}
\ee
where $q=q(x,z)\in (0, 1)$ and $C_4$ is some constant; in the last inequality we have used that $p\leq 2$.
Furthermore for $|x|\leq (1-\eps)^{-1}$ we have $1<|z|\leq\eps|x|<\eps(1-\eps)^{-1}$ and
\be
\label{e:V1}
|V(x+z)-V(x)-\nabla V(x)\cdot z|\leq C_9,\quad C_9=C_9(\e)>0.
\ee
Hence \eqref{bound_p<2}, \eqref{e:V1} and assumption \textbf{A$_{\nu,p}$} yield
\be\label{I1}
I_t^1\leq C_{10}, \quad C_{10}=C_{10}(\eps)>0.
\ee
Combining the  estimates \eqref{bound_drift}, \eqref{bound_diff}, \eqref{bound_small} for $a^D$, $a^M$, $a^N$ with the above estimates \eqref{bound_large}, \eqref{I1},  we get eventually that for $1<p\leq 2$ the inequality
\eqref{e:aV-1} still holds true, and for any $c_V<p\beta$ there exists $C_V$ large enough such that \eqref{L} holds.
\be
\label{e:bound1}
a^V_t\leq C-p\beta |X_t|^{p+\kappa-1}.
\ee

\smallskip

\noindent
\textbf{Case III: $p>2$.} Since $p>2>1$, we have $a_t=\frac{\di}{\di t}A_t^{\infty}$ and the canonical decomposition \eqref{eqinfty} should be used. The It\^o formula \eqref{Ito} and the decomposition \eqref{a} of the drift term $a^V$ remain the same, but the estimates should be properly changed. Namely, we will see that, instead of \eqref{e:bound1}, the following bound holds:
\be
\label{e:bound2}
a^V_t\leq C+C_*|X_t|^{p-2}-p\beta |X_t|^{p+\kappa-1}.
\ee
This explains the dichotomy between the cases $\kappa>-1$  and  $\kappa=-1$: in the first one the
(negative) drift term $-p\beta |X|^{p+\kappa-1}$ dominates the (positive) terms $C$, $C_*|X_t|^{p-2}$, while in the second case,
to get such a domination, we need to compare the constants $C_*$, $-p\beta$ which are multiplied by the same term $|X_t|^{p-2}$. Thus we redo, with proper changes, the above estimates for the terms $a^D$, $a^{M}$, $a^{N}$, $a^J$, paying the extra attention to the constants.

\noindent
\textbf{1. Term $a^D$:} the estimates are literally the same as in Case I,  and \eqref{bound_drift} holds.

\noindent
\textbf{2. Term $a^M$:} the first two lines in \eqref{e:VTaylor} remain true, while the last two fail because now $p>2$. From this first half of the
 \eqref{e:VTaylor} we get the following analogue of \eqref{bound_drift}:
 \be\label{bound_diff2}
a^M_t\leq C_2+ \frac{p}{2}\Big(c_{\T}+(p-2)c_{\langle M\rangle}\Big)|X_t|^{p-2}.
\ee

\noindent
\textbf{3. Term $a^N$:} the first two lines in \eqref{e:aN} remain true. From this first half of the
 \eqref{e:aN} we get the following analogue of \eqref{bound_small}: for any $\gamma>1$, there exists $C_4=C_4(\gamma)$ such that
 \be\label{bound_small2}
a^N_t\leq C_4+  \gamma\frac{p(p-1)}{2}|X_t|^{p-2}\int_{|z|\leq 1} |z|^2\, K_t(\di z).
\ee

\noindent
\textbf{4. Term $a^J$:} the decomposition \eqref{aJ} and the bound \eqref{bound_large} for $I_2$ remain the same; recall that $\eps\in (0,1)$  is a parameter. To estimate $I_1$, write for $p>2$, $|z|\leq \eps |x|$ and $|x|>(1-\eps)^{-1}$ by the Taylor formula \eqref{Taylor}
\be
\begin{aligned}
V(x+z)-V(x)-\nabla V(x)\cdot z
&=\frac{p}{2}\Big(|x+q z|^{p-2}|z|^2+(p-2)|x+q z|^{p-4}(x\cdot z+q |z|^2)^2\Big)\\
&=\frac{p}{2}|x|^{p-2}|z|^2 \Big(\Big|1+q \frac{z}{|x|}\Big|^{p-2}
+(p-2)\Big|1+q \frac{z}{|x|}\Big|^{p-4}\Big(\frac{x\cdot z}{|x||z|}+q \frac{|z|}{|x|}\Big)^2\Big)\\
&\leq (1+\e)^{p-2}\frac{p}{2}|x|^{p-2}|z|^2 \Big(1 +(p-2) \Big(\frac{1+\e}{1-\e} \Big)^2\Big)\\
&\leq \frac{(1+\e)^p}{(1-\e)^{p-2}} \frac{p(p-1)}{2}|x|^{p-2} |z|^2.  \\
\end{aligned}
\ee
This leads to the following analogue of \eqref{I1}:
\be\label{I12}
I_t^1\leq C_{10}+ \frac{(1+\e)^p}{(1-\e)^{p-2}} \frac{p(p-1)}{2}|X_t|^{p-2}  \int_{|z|> 1} |z|^2\, K_t(\di z).
\ee

We have
\be
\gamma\int_{|z|\leq 1} |z|^2\, K_t(\di z)+ \frac{(1+\e)^p}{(1-\e)^{p-2}}\int_{|z|> 1} |z|^2\, K_t(\di z)\leq \max\left(\gamma, \frac{(1+\e)^p}{(1-\e)^{p-2}}\right)\int_{\bR^n} |z|^2\, K_t(\di z).
\ee
Hence, summarising \eqref{bound_drift}, \eqref{bound_diff2}, \eqref{bound_small2}, \eqref{bound_large}, and \eqref{I12} we get that, for every $\gamma>1, \eps\in (0,1)$ there exists a constant $C=C(\gamma, \eps)$ such that \eqref{e:bound2} holds true with
\be
C_*=C_*(\gamma, \eps)=\frac{p}{2}\Big(c_{\T}+(p-2)c_{\langle M\rangle}\Big)+\frac{p}{2}(p-1)\max\left(\gamma, \frac{(1+\e)^p}{(1-\e)^{p-2}}\right)c_{\langle N\rangle}
\ee
Now we can complete the entire proof. For $\kappa>-1$, take \eqref{e:bound2} with any fixed $\gamma, \eps$. Since the term  $-p\beta |X|^{p+\kappa-1}$  dominates the terms $C, C_*|X_t|^{p-2}$, this yields for any $c_V<p\beta$ there exists $C_V$ large enough such that \eqref{L} holds. For $\kappa=-1$, note that
\be
C_*(1,0):=\lim_{\gamma\searrow1, \eps\searrow 0}C_*(\gamma, \eps)=\frac{p}{2}\Big(c_{\T}+(p-2)c_{\langle M\rangle}+(p-1)c_{\langle N\rangle}\Big),
\ee
and
\be
p<2+\frac{2\beta-c_{\T}-c_{\langle N\rangle}}{c_{\langle M\rangle}+c_{\langle N\rangle}}\quad\Leftrightarrow\quad
2\beta>(p-2)(c_{\langle M\rangle}+c_{\langle N\rangle})+c_{\T}+c_{\langle N\rangle} \quad\Leftrightarrow\quad p\beta> C_*(1,0).
\ee
Thus, under condition \eqref{cond_p}, we can fix $\gamma>1$ and $\eps>0$ such that
$
 p\beta- C_*(\gamma,\eps)>0.
$
Then for any $c_V< p\beta- C_*(\gamma,\eps)$ there exists $C_V$ large enough such that \eqref{L} holds.

\hfill $\square$

\section{Proof of Theorem \ref{t1}\label{s:t1}}

\subsection{Preamble: the proof in the standard Lyapunov case $\gamma=1$.}\label{s41}

The proof in the case $\gamma=1$ is simple and standard but we sketch it here
for the benefit of the reader. Let $V$ be a standard Lyapunov function, then by the It\^o formula
applied to the function $H(t,v)=\ex^{ct}v$  one obtains that for an arbitrary stopping time $\sigma$ the process
\be
H_t^{\sigma,V}=\ex^{c(t\vee \sigma-\sigma)}V_t
\ee
is a semimartingale with the decomposition
\be
\di H_t^{\sigma,V}=a_t^{H,\sigma,V}\, \di t+\di M_t^{H, \sigma,V}.
\ee
The process $M_t^{H, \sigma,v}$ is a local supermartingale, and the drift satisfies
\be
a_t^{H, \sigma,v}=a^V_t\quad \text{on}\quad \{t\leq \sigma\}
\ee
and
\be
\label{e:aH}
a_t^{H, \sigma,v} \leq  \ex^{c(t-\sigma)}\Big(C_V+(c-c_V)V_t\Big)\quad  \text{on}\quad \{t> \sigma\}.
\ee
Let $c<c_V$, then the two latter inequalities yield
\be
\label{in1}
a_t^{H, \sigma,v} \leq C_V  \ex^{c(t  -\sigma)}\quad  \text{on}\quad \{t> \sigma\}.
\ee
Now, we perform the ``time localization$+$Fatou lemma trick''. Namely, let $\tau_n\nearrow\infty$ be a localizing sequence of stopping times for $M_t^{H, \sigma,v}$. Then for each $n\in \mathbb{N}$ on the event $\{t> \sigma\}\cap\{\tau_{n}\geq \sigma\}$
\be
\E_\sigma(\ex^{c(t\wedge \tau_n-\sigma\wedge \tau_n)} V_{t\wedge\tau_n}-V_{\sigma\wedge\tau_n})=\E_\sigma H_{t\wedge \tau_n}^{\sigma,V}-H_{\sigma\wedge \tau_n}^{\sigma,V}\leq C_V \E_\sigma\int_{\sigma\wedge\tau_n}^{t\wedge\tau_n}   \ex^{c(s-\sigma)}\, \di s
\leq \frac{C_V}{c}\Big(\ex^{c(t-\sigma)}-1\Big),
\ee
where we have used that  $H_{\sigma\wedge \tau_n}^{\sigma,V}$ is $\rF_{\sigma\wedge \tau_n}$-measurable and
$\rF_{\sigma\wedge \tau_n}\subset \rF_\sigma.$
Then for any fixed  $n_0\geq 1$ and $n\geq n_0$ we have on the event $\{t>\sigma\}\cap\{\tau_{n_0}\geq \sigma\}$
\be
\E_\sigma(\ex^{c(t\wedge \tau_n-\sigma\wedge \tau_n)} V_{t\wedge\tau_n}\leq V_\sigma+\frac{C_V}{c}\Big(\ex^{c(t-\sigma)}-1\Big).
\ee
While $n\to \infty$, the  expression under the expectation on the left hand side tends to
$\ex^{c(t-\sigma)} V_t$ a.s.\ and it is positive. Hence using Fatou's lemma we get
\be
\E_\sigma \ex^{c(t-\sigma)} V_t\leq V_\sigma + \frac{C_V}{c}\Big(\ex^{c(t-\sigma)}-1\Big)
\quad  \text{on}\quad \{t>\sigma\}\cap\{\tau_{n_0}\geq \sigma\}.
\ee
Since $\{\tau_{n_0}\geq \sigma\}\uparrow \Omega$ as $n_0\to \infty$, we have the previous inequality actually valid a.s.\
on $\{t>\sigma\}$.  Dividing its both sides by an $\rF_\sigma$-measurable variable $\ex^{c(t-\sigma)}$, we get  \eqref{moment_sub-V} for $\gamma=1$.

To estimate exponential moments of the return time $\delta^\sigma_R$ (see \eqref{PT_lin} for the definition), we note that the we have actually shown in the  proof of Theorem \ref{tL} that under the assumptions \textbf{A} the Lyapunov condition \eqref{L} holds true with arbitrary $c_V<p\beta$.
Hence for any  $c<p\beta$ we can choose $c_V\in (c, p\beta)$ and
$R>1$ large enough such that on the event $\{t> \sigma\}$
\be
|X_t|>R  \quad \Rightarrow\quad a_t^{H, \sigma,v}\leq 0.
\ee
Using the same ``time localization$+$Fatou's lemma trick'' as above we obtain
\be
\begin{aligned}
\E_\sigma \ex^{c\delta_R^\sigma}&\leq \liminf_{m\to \infty} \E_\sigma \ex^{c(\delta_R^\sigma\wedge m)}
\leq \liminf_{m\to \infty}\E_\sigma H_{\tau_R^\sigma\wedge (m+\sigma)}^{\sigma,V}
\leq V_\sigma\leq |X_\sigma|^p+C,
\end{aligned}
\ee
where the last inequality holds true again because $V(x)-|x|^p$ is bounded by 2.
This proves \eqref{PT_lin}. \hfill $\square$

\subsection{The sub-Lyapunov case $\gamma\in (0,1)$: the supermartingale property of a flow transformed semimartingale}

In the case $\gamma\in (0,1)$, we use principally the same idea to transform the process $V_t$ into a semimartingale using a
properly chosen function $H(t,v)$. This idea is not genuinely new, see e.g.\ \cite[Theorem 4.1(i)]{DoucFortGui2009} or \cite[Section 4.1.2]{hairer2016}.
For the benefit of the reader, we explain  in details the way this idea is implemented here.

An informative analogy to what will be made below is provided by ODEs. Namely, let $\phi=\phi(v)$ be a smooth real valued function such the solutions of the ODE $v'=-\phi(v)$, $v_0=x$, determine a
flow of homeomorphisms. The inverse flow $H=H(t,x)$ that satisfies the reversed time ODE $H'=\phi(H)$
straightens up the flow $v(x)$, i.e.\ $\partial_t H(t,v_t(x))\equiv 0$. Considering $v_t$ as a (non-random) semimartingale we can say that the process $H(t,v_t)$ has zero drift. We are going to extend this observation to a positive semimartingale $V=(V_t)_{t\geq 0}$ whose drift has a bound
$a_t^{V}\leq -\phi(V_t)$ with
some positive concave function $\phi$ increasing to infinity. We will show that the
deterministic flow $H$ generated by the function $\phi$ transforms the semimartingale $V$ into a supermartingale.

\begin{prp}
\label{p2}
Let $\phi\in C([0,\infty),\bR_+)\cap C^2((0, \infty),\bR_+)$ be a concave function with $\phi(0)=0$, $\phi(\infty)=\infty$.
Let  $V=(V_t)_{t\geq 0}$ be a positive semimartingale with decomposition \eqref{DM}. Assume that there are $V_*>0$ and $c_*>0$ such that
\be\label{boundV}
a_t^{V}\leq -c_*\phi(V_t) \quad\hbox{ whenever}\quad V_t>V_*.
\ee
Let also $H=H(t,v)$ be the solution to the Cauchy problem
\be
\label{e:Cauchy}
H'_t=c_*\phi(H), \quad H(0,v)=v.
\ee
Then for any stopping time $\sigma$, the process
$H_t^{\sigma,V}:=H(t\vee\sigma-\sigma, V_{t})$, $t\geq 0$, is a semimartingale with decomposition
\be
\di H_t^{\sigma,V}= a_t^{H,\sigma,V}\, \di t +\di M^{H,\sigma,V}_t
\ee
and
\be
\label{boundH}
a_t^{H,\sigma,V}\leq 0 \quad\text{whenever}\quad V_t>V_*.
\ee
\end{prp}
\begin{proof} The differential equation \eqref{e:Cauchy} can be solved explicitly. Denote
\be
\Phi(v)=\int_1^v\frac{\di w}{\phi(w)}, \quad v>0.
\ee
Then $H$ is given by
\be
\label{H}
H(t,v)=\Phi^{-1}(c_*t+\Phi(v)), \quad t\geq 0, \, \, v>0.
\ee
Because $\phi$ is concave it is sub-linear. Hence the range of $\Phi$, or equivalently the domain of $\Phi^{-1}$, equals $(\varrho,+\infty)$ with
\be
\varrho= -\int_0^1\frac{\di w}{\phi(w)}.
\ee
The function $H $ is well defined for $t\in [0, \infty)$, $v\in (0, \infty)$, and  it is easy to check
that it is $C^1$ in $t$ and $C^2$ in $v$. Then we can apply the It\^o formula to get
on the event $\{t\geq \sigma\}$ that
\be
\label{e:H-Ito}
\begin{aligned}
H^{\sigma,V}(t-\sigma,V_t)&=V_\sigma+ \int_{\sigma}^t \prt_t H^{\sigma,V}(s-\sigma, V_s)\, \di s
+ \int_\sigma^t \prt_v H^{\sigma,V}(s-\sigma, V_{s-})\, \di V_s\\
&+ \frac12 \int_\sigma^t \prt_{vv}^2 H^{\sigma,V}(s-\sigma, V_s)\,\di\langle M^{V,c}\rangle_s
\\&  + \sum_{\sigma < s \leq t}\Big[ H^{\sigma,V}(s-\sigma , V_s)- H(s-\sigma , V_{s-})-\prt_v H^{\sigma,V}(s-\sigma , V_{s-})\Delta V_s\Big].
\end{aligned}
\ee
Since $t\mapsto H(t,v)\geq  0$ is increasing and $\phi'$ is decreasing we have
\be
\prt_{vv}^2 H(t,v)=\frac{\phi'(H(t,v))\phi(H(t,v))-\phi(H(t,v))\phi'(v)}{\phi^2(v)}\leq 0,
\ee
Hence the last two lines in \eqref{e:H-Ito} are non-positive. Therefore, taking into account the inequality $\prt_v H(t, v)>0$,
we get the required semimartingale decomposition for $H^{\sigma,V}$ with
\be
\begin{aligned}
a_t^{H, \sigma,V}& =a_t^V\quad \text{for }t<\sigma, \\
a_t^{H, \sigma,V}& \leq \prt_t H(t-\sigma, V_t)+\prt_v H(t-\sigma, V_t)a_t^V \quad \text{for }t\geq \sigma.
\end{aligned}
\ee
Recall that $\prt_t H=c_*\phi(H)=c_*\phi(v)\prt_v H$, and  \eqref{boundV} holds. Hence, whenever $V_t\geq V_*$,
\be
\begin{aligned}
a_t^{H, \sigma,V}&\leq -c_*\phi(V_t)\leq 0 \quad \text{for }t<\sigma, \\
 a_t^{H, \sigma,V}&\leq  \prt_t H(t-\sigma, V_t)-c_*\prt_v H(t-\sigma, V_t)\phi(V_t)=0 \quad \text{for }t\geq \sigma.
\end{aligned}
\ee
\end{proof}

\begin{rem}
\label{r:H}
In what follows we will take $\gamma\in (0,1)$ and
\be
\label{e:phi}
\phi(v)=v^{\gamma},\quad v\geq 0.
\ee
Then by the Lyapunov condition \textbf{L}$_{V, \gamma}$, for each $c_*\leq c_V$ there exists $V_*>0$ such that \eqref{boundV} holds. Without loss of generality we will assume that $V_*\geq 1$. The corresponding function $H$ has the form
\be
\label{e:H}
H(t,v)=\Big((1-\gamma)c_*t+v^{1-\gamma}\Big)^\frac{1}{1-\gamma},\quad t\geq 0, \ v\geq 0.
\ee
\end{rem}

\subsection{The sub-Lyapunov case $\gamma\in (0,1)$: estimates involving passage-times.}
 In this short section we derive several corollaries from Proposition \ref{p2} for moments involving passage times.
 These estimates will be used in the proof of the main moment estimate \eqref{moment_sub-V} of  Theorem \ref{t1}. They also will prove
 \eqref{PT_sub}. By analogy with \eqref{e:tauR}, we denote
\be
\label{e:tauVQ}
\tau_Q^{V,\sigma}=\inf\{t\geq \sigma\colon V_t\leq Q\}.
\ee
Recall that for the Lyapunov function $V$ from Theorem \ref{tL} we have $V(x)=|x|^p, |x|\geq 1$, and thus for $Q\geq 1$
\be\label{e:tauVQ2}
\tau_Q^{V, \sigma}=\tau_{Q^{1/p}}^\sigma.
\ee

\begin{cor}
\label{c1}
Let  the Lyapunov condition \emph{\textbf{L}}$_{V, \gamma}$ hold and $c_*$, $V^*$ be chosen as in Remark \ref{r:H}.
Then for any $Q\geq V^*$ and  stopping time $\sigma$,  we have a.s.\ on the event $\{t>\sigma\}$
\be
\label{b3}
\E_{\sigma}V_t\bI_{\tau^{V,\sigma}_{Q}>t}\leq V_\sigma
\ee
and
\be\label{b4}
\P_{\sigma}(\tau^{V,\sigma}_{Q}>t)\leq  \Big((1-\gamma)c_*\Big)^{-\frac{1}{1-\gamma}}\cdot (t-\sigma)^{-\frac{1}{1-\gamma}}\cdot V_\sigma.
\ee
\end{cor}
\begin{proof}
It follows from Proposition \ref{p2} and the ``localization$+$Fatou's lemma'' argument  that on the event $\{t>\sigma\}$
\be
\label{b41}
\E_{\sigma} H(t\wedge\tau^{V,\sigma}_{Q}-\sigma, V_{t\wedge\tau^{V,\sigma}_{Q}})\leq V_\sigma.
\ee
Since $H(t,v)\geq 0$ and  $t\mapsto H(t,v)$ is increasing, on $\{t>\sigma\}$ we have
\be
\E_{\sigma}V_t\bI_{\tau^{V,\sigma}_{Q}>t}=\E_{\sigma} \Big[H(0, V_t)\bI_{\tau^{V,\sigma}_{Q}>t}\Big]\leq \E_{\sigma} \Big[H(t-\sigma, V_t)\bI_{\tau^{V,\sigma}_{Q}>t}\Big]\leq V_\sigma,
\ee
which is just \eqref{b3}.

Next, since $V_t\geq 0$ and $v\mapsto H(t,v)$ is increasing we have on $\{t>\sigma\}$ that
\be
H(t-\sigma, V_t)\bI_{\tau^{V,\sigma}_{Q}>t}\geq H(t-\sigma, 0)\bI_{\tau^{V,\sigma}_{Q}>t}=\Big((1-\gamma)c_*(t-\sigma)\Big)^\frac{1}{1-\gamma}\bI_{\tau^{V,\sigma}_{Q}>t}.
\ee
Since $t-\sigma$ is $\rF_\sigma$-measurable, we obtain from  this inequality and \eqref{b41}
\be
\label{b5}
\E_\sigma\Big[\Big((1-\gamma)c_*(t-\sigma)\Big)^\frac{1}{1-\gamma}\bI_{\tau^{V,\sigma}_{Q}>t}\Big]
=\Big((1-\gamma)c_*(t-\sigma)\Big)^\frac{1}{1-\gamma}\P_\sigma(\tau^{V,\sigma}_{Q}>t)\leq V_\sigma
\ee
which immediately implies \eqref{b4}.
\end{proof}

By a slight change of the proof of the above corollary, we get the moment bound for the passage-time stated in Theorem \ref{t1}.

\begin{cor}[estimate \eqref{PT_sub} in Theorem \ref{t1}]
\label{c11}  Let the assumptions of statement (ii) of Theorem \ref{t1} hold, then by Theorem \ref{tL}  the Lyapunov condition
\emph{\textbf{L}}$_{V, \gamma}$  holds true with $V$ satisfying \eqref{e:V}, and
$$
\gamma=\frac{p+\kappa-1}{p}.
$$
Let  $c_*$ and $V^*$ be chosen
in Remark \ref{r:H} and let $R\geq V_*^{1/p}$. Then for any stopping time $\sigma$,
$$
\E_{\sigma}(\delta_R^\sigma)^\frac{p}{1-\kappa}\leq  \Big(\frac{1-\kappa}{p}c_*\Big)^{-\frac{p}{1-\kappa}} |X_\sigma|^p.
$$
\end{cor}

\begin{proof}
By \eqref{e:tauVQ2}, we have similarly to \eqref{b5} that for any $t>0$
\be
\Big(\frac{1-\kappa}{p}c_*\Big)^\frac{p}{1-\kappa}
\E_{\sigma}\Big(t\wedge\tau_R^\sigma-\sigma\Big)^\frac{p}{1-\kappa}
=\E_{\sigma} H(t\wedge\tau_R^\sigma-\sigma, 0)\leq \E_{\sigma} H(t\wedge\tau_R^\sigma-\sigma, V_{t\wedge\tau_R^\sigma})\leq V_\sigma=|X_\sigma|^p
\ee
on the event $\{t>\sigma\}$. Taking $t\to \infty$, we get \eqref{PT_sub} by Fatou's lemma.
\end{proof}

Finally, we apply the H\"older inequality in order to get the following bound from \eqref{b3} and \eqref{b4}: for any $0<\gamma'<1$,
on the event $\{t>\sigma\}$
\be
\label{b61}
\begin{aligned}
\E_{\sigma}V_t^{\gamma'}\bI_{\tau^{V,\sigma}_Q>t}
&\leq \Big(\E_{\sigma}\bI_{\tau^{V,\sigma}_Q>t}\Big)^{1-\gamma'} \cdot   \Big(\E_{\sigma}V_t\bI_{\tau^{V,\sigma}_Q>t}\Big)^{\gamma'}\\
&\leq \Big[\Big((1-\gamma)c_*\Big)^{-\frac{1}{1-\gamma}}\cdot (t-\sigma)^{-\frac{1}{1-\gamma}}\cdot V_\sigma\Big]^{1-\gamma'}
\cdot \Big[ V_\sigma\Big]^{\gamma'}\\
&=\Big((1-\gamma)c_*\Big)^{-\frac{1-\gamma'}{1-\gamma}}\cdot (t-\sigma)^{-\frac{1-\gamma'}{1-\gamma}} \cdot V_\sigma.
\end{aligned}
\ee

\subsection{The sub-Lyapunov case $\gamma\in (0,1)$: completion of the proof.}

Now  we can proceed with the proof of  the  main moment estimate \eqref{moment_sub-V} of  Theorem \ref{t1}.
Let $c_*$, $V^*$ be chosen in Remark \ref{r:H} and $Q\geq V^*$ be fixed.

Let $\gamma'<\gamma<1$.
On the set $\{t>\sigma\}$, we have
\be
\label{b0}
\E_\sigma V_t^{\gamma'}=\E_\sigma V_t^{\gamma'}\bI_{\tau^{V,\sigma}_Q>t}+\E_\sigma V_t^{\gamma'}\bI_{\tau^{V,\sigma}_Q\leq t}.
\ee
Then for the first term we have simply by the Jensen inequality and \eqref{b3} that
\be
\label{b01}
\E_{\sigma}V_t^{\gamma'}\bI_{\tau^{V,\sigma}_Q>t}
\leq \Big(\E_{\sigma}V_t\bI_{\tau^{V,\sigma}_Q>t}\Big)^{\gamma'}
\leq V_\sigma^{\gamma'}.
\ee
The estimate for the second term is based on the following lemma, which combines the moment bounds \eqref{b61} (valid before the passage-time $\tau^{V,\sigma}_Q$) and \eqref{moment_finite} (informative on bounded time intervals, only) using a  renewal argument.
\begin{lem}
\label{l1}
There exist constants $C_0, C_1>0$ such that for  any  $k\geq 2$,  any stopping time $\sigma$, and any $t\geq 0$
\be
\label{b7}
\E_{{\tau^{V,\sigma}_Q}} V_t^{\gamma'}
\leq C_0+C_1\sum_{j=1}^{k-2}j^{-\frac{1-\gamma'}{1-\gamma}}
\quad  \text{ on the event } \quad \{\tau^{V,\sigma}_Q\leq t\}\cap \{t\leq \tau^{V,\sigma}_Q+k\}\in \rF_{\tau^{V,\sigma}_Q}.
\ee
\end{lem}
\begin{proof}
By  \eqref{moment_finite}  we have for any stopping time $\theta$ on the set $\{t\geq \theta\}$
\be
\E_{\theta} V_t\leq V_\theta+C_V(t-\theta).
\ee
We apply this inequality with $\theta=\tau^{V,\sigma}_Q$.
Note that in this case $V_\theta=V_{\tau^{V,\sigma}_Q}\leq Q$ because the process $V$ has c\`adl\`ag trajectories.
Hence
\be
\E_{{\tau^{V,\sigma}_Q}} V_t\leq C \quad \hbox{ on }\quad A_2(t):=\{t\geq\tau^{V,\sigma}_Q\}\cap \{t\leq \tau^{V,\sigma}_Q+2\}.
\ee
for some $C>0$.

By the H\"older inequality this proves \eqref{b7} for $k=2$ with a proper constant $C_0>0$.
Taking $C_0$ large enough, one can also  guarantee by essentially the same argument that for any stopping time $\varsigma\geq \tau^{V,\sigma}_Q$ with $\varsigma\leq \tau^{V,\sigma}_Q+1,$
\be\label{b9}
\E_{\tau^{V,\sigma}_Q} V_\varsigma\leq C_0.
\ee

To prove the entire  bound  \eqref{b7}, we use induction by $k$. The base $k=2$ is just verified. For $k\geq 3$,
assume the required bound to be true for $k-1$ with $C_0$ as above and
\be\label{C_1}C_1= C_0 \cdot \Big((1-\gamma)c_*\Big)^{-\frac{1-\gamma'}{1-\gamma}}.
\ee
Then on the set
$\{t\geq\tau^{V,\sigma}_Q\}\cap \{t-\tau^{V,\sigma}_Q\leq k-1\},$ we have
\be\label{k-1}
  \E_{\tau^{V,\sigma}_Q}V_t^{\gamma'}\leq C_0+C_1\sum_{j=1}^{k-3}j^{-\frac{1-\gamma'}{1-\gamma}}.
\ee
Next, denote $\theta=\tau^{V,\sigma}_Q+1$ and write
  \be\label{b8}\ba
  \E_{{\tau^{V,\sigma}_Q}}V_t^{\gamma'}&=\E_{{\tau^{V,\sigma}_Q}}\Big[\bI_{t\geq\tau^{V,\theta}_Q} \E_{{\tau^{V,\sigma}_Q}}V_t^{\gamma'}
  +\bI_{t<\tau^{V,\theta}_Q}\E_{{\tau^{V,\theta}_Q}}V_t^{\gamma'}\Big].
     \ea
  \ee
  Since $\tau^{V,\theta}_Q\geq \theta=\tau^{V,\sigma}_Q+1$, we have for $k\geq 3$
\be
A_k(t):=\{t\geq\tau^{V,\sigma}_Q\}\cap \{t-\tau^{V,\sigma}_Q\in (k-1, k]\}=\{t-\tau^{V,\sigma}_Q\in (k-1, k]\}\subset \{t-\tau^{V,\theta}_Q\leq k-1\}.
\ee
Then by the assumption of the induction applied to $\sigma:=\theta$
\be
\bI_{A_k(t)}\bI_{t\geq\tau^{V,\sigma}_Q}\E_{\tau^{V,\sigma}_Q}V_t^{\gamma'}
\leq \bI_{t\geq \tau^{V,\theta}_Q}\bI_{t\leq \tau^{V,\theta}_Q+k-1}\E_{\tau^{V,\theta}_Q}V_t^{\gamma'}
\leq C_0+C_1\sum_{j=1}^{k-3}j^{-\frac{1-\gamma'}{1-\gamma}}.
\ee
Since $A_k(t)$ is $\rF_{\tau^{V,\sigma}_Q}$-measurable, this  bounds the first term in \eqref{b8}  on the event $A_k(t)$:
\be
\label{e:1term}
\begin{aligned}
\E_{{\tau^{V,\sigma}_Q}}\Big[\bI_{t\geq\tau^{V,\sigma}_Q} \E_{{\tau^{V,\sigma}_Q}}V_t^{\gamma'}\Big]
&=\bI_{A_k(t)}\E_{{\tau^{V,\sigma}_Q}}\Big[\bI_{t\geq\tau^{V,\theta}_Q} \E_{{\tau^{V,\theta}_Q}}V_t^{\gamma'}\Big]\\
&=\E_{{\tau^{V,\sigma}_Q}}\Big[\bI_{A_k(t)}\bI_{t\geq\tau^{V,\theta}_Q} \E_{{\tau^{V,\theta}_Q}}V_t^{\gamma'}\Big]
\leq C_0+C_1\sum_{j=1}^{k-3}j^{-\frac{1-\gamma'}{1-\gamma}}.
\end{aligned}
\ee
To estimate the second term, we note that the event $A_k(t)$ belongs to $\rF_{\tau^{V,\sigma}_Q}\subset \rF_\theta$, thus
 on this set
\be
\begin{aligned}
\E_{{\tau^{V,\sigma}_Q}}\Big[\bI_{t<\tau^{V,\theta}_Q}\E_{{\tau^{V,\theta}_Q}}V_t^{\gamma'}\Big]
&=\E_{{\tau^{V,\sigma}_Q}}\Big[V_t^{\gamma'}\bI_{t<\tau^{V,\theta}_Q}\Big]
=\E_{{\tau^{V,\sigma}_Q}}\Big[V_t^{\gamma'}\bI_{A_k(t)}\bI_{t<\tau^{V,\theta}_Q}\Big]
=\E_{{\tau^{V,\sigma}_Q}}\Big[\bI_{A_k(t)} \E_{{\theta}}V_t^{\gamma'}\bI_{t<\tau^{V,\theta}_Q}\Big].
\end{aligned}
\ee
In addition, on the event $A_k(t)$ one has $t-\theta=t-\tau^{V,\sigma}_Q-1>k-2$, and thus by \eqref{b61} with the stopping time $\theta$ instead of $\sigma$,
\be
\begin{aligned}
\E_{\tau^{V,\sigma}_Q}\Big[\bI_{A_k(t)} \E_\theta\Big[ V_t^{\gamma'}\bI_{t<\tau^{V,\theta}_Q}\Big]\Big]
&=\E_{{\tau^{V,\sigma}_Q}}\Big[\bI_{A_k(t)}\bI_{t>\theta} \E_\theta\Big[|X_t|^{p'}\bI_{t<\tau^{V,\theta}_Q}\Big]\Big]\\
&\leq  \Big((1-\gamma)c_*\Big)^{-\frac{1-\gamma'}{1-\gamma}}
\cdot \E_{\tau^{V,\sigma}_Q}\Big[\bI_{A_k(t)}\cdot (t-\theta)^{-\frac{1-\gamma'}{1-\gamma}} \cdot V_\theta\Big]
\\&\leq   \Big((1-\gamma)c_*\Big)^{-\frac{1-\gamma'}{1-\gamma}} \cdot (k-2)^{-\frac{1-\gamma'}{1-\gamma}} \cdot \E_{\tau^{V,\sigma}_Q} V_\theta.
\end{aligned}
\ee
Using \eqref{b9} with $\varsigma=\theta=\tau^{V,\sigma}_Q+1$, we get the bound for the second term in \eqref{b8}  on the event $A_k(t)$:
\be
\E_{\tau^{V,\sigma}_Q}\Big[\bI_{t<\tau^{V,\theta}_Q}\E_{{\tau^{V,\theta}_Q}}V_t^{\gamma'}\Big]
\leq  \Big((1-\gamma)c_*\Big)^{-\frac{1-\gamma'}{1-\gamma}} \cdot (k-2)^{-\frac{1-\gamma'}{1-\gamma}} \cdot C_0 .
\ee
Since the constant $C_1$ was chosen as in \eqref{C_1}, this inequality combined with the bound \eqref{e:1term} for the first term in \eqref{b8}  gives  that on the event $A_k(t)$
\be\label{k}
\begin{aligned}
\E_{\tau^{V,\sigma}_Q}V_t^{\gamma'}&
\leq C_0+C_1\sum_{j=1}^{k-3}j^{-\frac{1-\gamma'}{1-\gamma}}
+ C_0 \cdot \Big((1-\gamma)c_*\Big)^{-\frac{1-\gamma'}{1-\gamma}} \cdot (k-2)^{-\frac{1-\gamma'}{1-\gamma}}
 \\&
=C_0+C_1\sum_{j=1}^{k-2}j^{-\frac{1-\gamma'}{1-\gamma}}.
\end{aligned}
\ee
Note that the event
\be
B_k(t):=\{\tau^{V,\sigma}_Q\leq t\}\cap \{t\leq \tau^{V,\sigma}_Q+k\}
\ee
can be represented as a disjoint union $B_k(t)=B_{k-1}(t)\cup A_k(t)$
of two $\rF_{\tau^{V,\sigma}_Q}$-measurable events $B_{k-1}(t)$ and $A_k(t)$.
Since on these two events $\E_{\tau^{V,\sigma}_Q}V_t^{\gamma'}$ admits the estimates \eqref{k-1} and \eqref{k} respectively,
 the proof of the induction step and thus of Lemma \ref{l1} is complete.
\end{proof}
Now we can finalize the proof of the main moment estimate \eqref{moment_sub-V} of  Theorem \ref{t1}. Since $\gamma'<\gamma$, we have
\be
C_*:=C_0+C_1\sum_{j=1}^\infty j^{-\frac{1-\gamma'}{1-\gamma}} <\infty.
\ee
By  Lemma \ref{l1}, we have on each of the sets $B_k(t)$, $k\geq 2$
\be
\E_{\tau^{V,\sigma}_Q}V_t^{\gamma'}\leq C_0+C_1\sum_{j=1}^{k-2}j^{-\frac{1-\gamma'}{1-\gamma}}    \leq C_* .
\ee
This yields
\be
\E_{\tau^{V,\sigma}_Q}V_t^{\gamma'}\leq C_* \quad \text{ on the event }\quad \bigcup_{k\geq 2}B_k(t)=\{\tau^{V,\sigma}_Q\leq t\}
\ee
and gives the required bound for the second term in \eqref{b0}:
\be
\E_\sigma\Big[V_t^{\gamma'}\bI_{\tau^{V,\sigma}_Q\leq t}\Big]
=\E_\sigma\Big[\bI_{\tau^{V,\sigma}_Q\leq t}\cdot \E_{\tau^{V,\sigma}_Q} V_t^{\gamma'}    \Big]\leq C_*\quad  \hbox{on}\quad \{t>\sigma\}.
\ee
Combining this bound with \eqref{b01} we get
\be
\E_\sigma V_t^{\gamma'}\leq V_\sigma^{\gamma'}+C_*\quad  \hbox{on}\quad \{t>\sigma\}.
\ee
\hfill $\blacksquare$

\section{Proofs: the case of  super-linear drifts}

In the case $\kappa>1$ we have $\gamma=\frac{p+\kappa-1}{p}>1$ and
Proposition \ref{p2} cannot be applied directly because the function $\phi(v)=v^\gamma$ is not concave. Because of that, we use a completely different argument, which we now outline.

In the deterministic setting, if a non-negative function $f$ satisfies
\be
f'_t\leq - \beta f_t^\kappa,
\ee
for some $\beta>0$
then the relaxation time from any positive starting point $f_0$ to the level $R>0$ is bounded from above by
\be
\inf\{t\geq 0\colon f_t\leq R\}\leq \frac{1}{\beta(\kappa-1)}R^{1-\kappa},
\ee
Note that hat this bound is uniform over all initial values $f_0\geq 0$.  To check this bound, one can simply consider the function $g_t=U(f_t)$ with the new Lyapunov function $U(f)=f^{1-\kappa}$, then
\be\label{aa1}
g'_t=(1-\kappa)f^{-\kappa}_t f'_t\geq \beta(\kappa-1),
\ee
and because $g_0\geq 0$ we have 
by integrating the bound \eqref{aa1} for $t\geq (\beta(\kappa-1))^{-1} R^{1-\kappa}$
\be
g_t\geq R^{1-\kappa}+g_0\geq R^{1-\kappa}
\ee
and hence by standard integral comparison results
\be
f_t\leq R.
\ee
The main idea of our proofs of Theorem \ref{t2} and Theorem \ref{t3} is to repeat, with proper changes, this simple argument in the stochastic setting. Namely, we will
consider the process $U_t=U(|X_t|)$  aiming to show that it is a semimartingale with
(a) the derivative of the predictable part dominated from below by a positive constant and
(b) the continuous martingale and jump parts that being negligible in comparison to the drift.
For the jump part to be negligible indeed, we have to exclude large jumps; that is, our construction will include a certain localization procedure.
Since $\kappa>1$ we will work with the drift $A^{\leq 1}$ due to Remark \ref{r1}. We further  proceed with details.

\subsection{Proof of Theorem \ref{t2}\label{s:t2}}

In what follows, $\kappa>1$ and stopping time $\sigma$ are fixed and $R>1$ is a parameter.
For a given $p_X\in (0, p+\kappa-1)$ we fix $\eps>0$ small enough such that $p_X<p+\kappa-1-p\eps$.
This $\eps$ is used to define the localization procedure mentioned above.  Namely, we put
\be
\varsigma^\sigma_R=\inf\{t>\sigma\colon  |\Delta X_t|>R^{1-\eps}\}
\ee
and define the process $X^R_t$ as $X_t$ before the stopping time $\varsigma^\sigma_R$, and a constant function  afterwards,
equal to the value $X$ \emph{prior to} a large jump:
\be
X_t^R=\begin{cases}
      X_t, & t<\varsigma^\sigma_R, \\
      X_{\varsigma^\sigma_R-}, & t\geq \varsigma^\sigma_R.
    \end{cases}
\ee
Then $X^R$ is a semimartingale with the representation
\be
\di X_t^R=\bI_{t\leq \varsigma^\sigma_R}\Big[a_t\, \di t+  \di M_t
+\int_{|z|\leq R^{1-\eps}}z\big(N(\di z, \di s)-\bI_{|z|\leq 1}\nu(\di z, \di s)\big)\Big]
\ee
for $p\in(0,1]$ and
\be
\di X_t^R=\bI_{t\leq \varsigma^\sigma_R}\Big[ a_t\, \di t+ \di M_t
+\int_{|z|\leq R^{1-\eps}}z\, \widetilde N(\di z, \di s)\Big].
\ee
for $p>1$, where $\widetilde N(\di z, \di s)= N(\di z, \di s)-\nu(\di z, \di s)$ is the compensated jump measure.

Denote  $\upsilon_R^\sigma=\varsigma^\sigma_R\wedge\tau_R^\sigma$ and put $U(x)=|x|^{1-\kappa}$.
Let us consider the process
\be
Y_t:=U(X^R_{t\wedge\upsilon_R^\sigma}),\quad t\geq 0.
\ee
Then by the It\^o formula
\be
\label{Y}
\di Y_t=\bI_{t\leq \upsilon_R^\sigma} a^Y_t\, \di t+ \di M_t^{Y,c}+\di M_t^{Y,d} ,
\ee
with the continuous- and jump- local martingale parts given by
\be
\begin{aligned}
\di M_t^{Y,c}&=\bI_{t\leq \upsilon_R^\sigma}  \nabla U(X_t) \cdot\, \di M_t,\\
\di M_t^{Y,d}&=\bI_{t\leq \upsilon_R^\sigma}  \int_{|z|\leq R^{1-\eps}}\Big(U(X_{t-}+z)-U(X_{t-})\Big)\wt N(\di z, \di t)
\end{aligned}
\ee
and the drift part
\be
a^Y=a^D+a^M+a^J
\ee
where
\be
\begin{aligned}
a_t^D&= \nabla U(X_t)\cdot a_t,\\
a_t^M&=\frac12\operatorname{Trace}\Big(\nabla^2 U(X_t)\cdot B_t\Big),\\
a^J_t&=\int_{|z|\leq R^{1-\eps}}\Big[U(X_{t}+z)-U(X_{t})-\bI_{|z|\leq 1}\nabla U(X_t)\cdot z\Big]K_t(\di z)
\end{aligned}
\ee

Let us  verify that for $R$ large enough and $\sigma<t\leq \upsilon_R^\sigma$ the term $a_t^{D}$ is the principal
one in the above decomposition.

\noindent
\textbf{Term $a^D$.} We have by \textbf{A}$_{a,\kappa}$ as in \eqref{e:aD}, for $R\geq R_0$,
\be
\label{bad}
a_t^{D}= -(\kappa-1)|X_t|^{-1-\kappa} X_t\cdot a_t \geq \beta(\kappa-1)=:c_Y>0, \quad \sigma<t\leq \upsilon_R^\sigma.
\ee

\noindent
\textbf{Term $a^M$.}
By \textbf{A}$_M$, as in \eqref{e:aM}, for $R\geq R_0$,
\be
\label{ba0}
a_t^{M}\geq -\frac{\kappa-1}{2}c_B|X_t|^{-1-\kappa}\geq -CR^{-1-\kappa}, \quad \sigma<t\leq \upsilon_R^\sigma,
\ee
which is obviously negligible when compared to \eqref{bad}.

\noindent
\textbf{Term $a^J$.}
We will show that
\be
\label{ba1}
a_t^{J}\geq -CR^{-\kappa-(p-1)_+}, \quad \sigma<t\leq \upsilon_R^\sigma.
\ee
For that, observe first that for $|x|\geq R, |z|\leq 1$, the esimate \eqref{e:VTaylor} yields
\be
|U(X_{t}+z)-U(X_{t})- \nabla U(X_t)\cdot z|\leq C|x|^{-1-\kappa}|z|^2
\ee
for some $C>0$ which yields by \textbf{A}$_{\nu,\leq 1}$
\be
\int_{|z|\leq 1}\Big[U(X_{t}+z)-U(X_{t})- \nabla U(X_t)\cdot z \Big]K_t(\di z)
\geq -CR^{-1-\kappa}, \quad \sigma<t\leq \upsilon_R^\sigma.
\ee
To estimate the part of $a_t^{Y,jump}$ which corresponds to the integral over $1<|z|\leq R^{1-\eps}$, we use inequality
\be
\Big||x+z|^{1-\kappa}-|x|^{1-\kappa}\Big|\leq C|x|^{-\kappa}|z|,
\ee
and consider separately two cases: $p> 1$ and $p\in (0,1]$. In the first case, we just use \textbf{A}$_{\nu,p}$ to get
\be
\begin{aligned}
\int_{1<|z|\leq R^{1-\eps}}\Big[|X_{t}+z|^{1-\kappa}&-|X_{t}|^{1-\kappa}\Big]\,K_t(\di z)
\geq -C |X_{t}|^{-\kappa} \int_{|z|>1}|z|\,K_t(\di z)\\
&\geq -C |X_{t}|^{-\kappa} \int_{|z|>1} |z|^p\,K_t(\di z)\geq -CR^{-\kappa}, \quad \sigma<t\leq \upsilon_R^\sigma.
\end{aligned}
\ee
In the second case, $p\in(0,1]$ , we have
\be
\begin{aligned}
\int_{1<|z|\leq R^{1-\eps}}&\Big[|X_{t}+z|^{1-\kappa}-|X_{t}|^{1-\kappa}\Big]\,K_t(\di z)
\geq -C |X_{t}|^{-\kappa} \int_{1<|z|\leq R^{1-\eps}}|z|\, K_t(\di z)\\
&\geq -C |X_{t}|^{-\kappa} \int_{1<|z|\leq R} R^{1-p}|z|^p\,K_t(\di z)
\geq -CR^{-\kappa-p+1}, \quad \sigma<t\leq \upsilon_R^\sigma.
\end{aligned}
\ee

Combining these two cases, we get \eqref{ba1}.

By \eqref{bad}, \eqref{ba0}, \eqref{ba1} we get that there exists $R_1$ such that, for $R\geq R_1$,
\be
a_t^Y \geq \frac{c_Y}{2}, \quad \sigma<t\leq \upsilon_R^\sigma.
\ee
This inequality serves in our argument as an analogue of \eqref{aa1}. Namely, we have $Y_\sigma\geq 0$ and therefore
\be
\label{ba2}
Y_t\geq (t-\sigma)\frac{c_Y}{2}+M_t^{Y,c}-M_\sigma^{Y,c}+M_t^{Y,d}-M_\sigma^{Y,d}, \quad \sigma<t\leq \upsilon_R^\sigma.
\ee
The continuous- and the jump-martingale parts in the decomposition \eqref{Y} are negligible when compared with the predictable part
in the following sense.
\begin{lem}
\label{l:mart}
For any $T>0$ and $C_1>0$, there exist $C_2>0$ and $\gamma>0$ such that for all $R\geq 1$
\begin{align}
\label{ba3}
&\P_\sigma\Big(\max_{\sigma\leq t\leq \sigma+T}|M_t^{Y,c}-M_\sigma^{Y,c}|>C_1R^{1-\kappa}\Big)\leq  C_2\ex^{-\gamma R},\\
\label{ba4}
&\P_\sigma\Big(\max_{\sigma\leq t\leq \sigma+T}|M_t^{Y,d}-M_\sigma^{Y,d}|>C_1R^{1-\kappa}\Big)\leq  C_2\ex^{-\gamma R^\e}.
\end{align}
\end{lem}
\begin{proof}
The proofs are given in Appendices \ref{a:mart1} and \ref{a:mart2}.
\end{proof}

Now, recall that
\be
Y_t=|X_t|^{1-\kappa}\leq R^{1-\kappa}, \quad \sigma<t\leq \upsilon_R^\sigma.
\ee
Then by \eqref{ba2}, for $R\geq R_1$,
\be
\Big\{\upsilon_R^\sigma-\sigma\geq \frac{4}{c_Y}R^{1-\kappa}\Big\}
\subseteq\Big\{ \min_{\sigma\leq t\leq \sigma+4c_Y^{-1}R^{1-\kappa}} (M_{t}^{Y,c}-M_\sigma^{Y,c})
+ \min_{\sigma\leq t\leq \sigma+4c_Y^{-1}R^{1-\kappa}}( M_{t}^{Y,d}-M_{\sigma}^{Y,d})\leq -R^{1-\kappa}\Big\},
\ee
which by  \eqref{ba3} and \eqref{ba4} yields
\be
\P_\sigma\Big(\upsilon_R^\sigma-\sigma\geq \frac{4}{c_Y}R^{1-\kappa}\Big)\leq C\ex^{-cR^{-\eps}}
\ee
for some $c,C>0$. Changing the variables $\frac{4}{c_Y}R^{1-\kappa}\rightsquigarrow R^{1-\kappa}$, we get
\be
\label{ba50}
\P_\sigma\Big(\upsilon_R^\sigma-\sigma\geq R^{1-\kappa}\Big)\leq C\ex^{-cR^{-\eps}}
\ee
with properly changed constants $c,C>0$. Recall that $\upsilon_R^\sigma=\tau_R^\sigma\wedge \varsigma^\sigma_R$, so that
\be
\begin{aligned}
\P_\sigma (\tau_R^\sigma-\sigma\geq R^{1-\kappa} )
&=\P_\sigma (\tau_R^\sigma-\sigma\geq R^{1-\kappa}, \upsilon_R^\sigma=\tau_R^\sigma)
+\P_\sigma (\tau_R^\sigma-\sigma\geq R^{1-\kappa}, \upsilon_R^\sigma=\varsigma_R^\sigma )\\
&\leq \P_\sigma (\upsilon_R^\sigma-\sigma\geq R^{1-\kappa} )\\
&+\P_\sigma (\tau_R^\sigma-\sigma\geq R^{1-\kappa}, \varsigma_R^\sigma-\sigma\geq  R^{1-\kappa}   ,\varsigma_R^\sigma\leq \tau_R^\sigma )\\
&+\P_\sigma (\tau_R^\sigma-\sigma\geq R^{1-\kappa}, \varsigma_R^\sigma-\sigma<  R^{1-\kappa}   ,\varsigma_R^\sigma\leq \tau_R^\sigma )\\
&\leq 2\P_\sigma (\upsilon_R^\sigma-\sigma\geq R^{1-\kappa} )
+\P_\sigma ( \varsigma_R^\sigma-\sigma<  R^{1-\kappa} ).
\end{aligned}
\ee
It is easy to show that
\be
\label{ba5}
\P_\sigma(\varsigma_R^\sigma-\sigma< R^{1-\kappa} )\leq c_{\nu,p} R^{1-\kappa-p+p\eps},
\ee
see Appendix \ref{a:varsigma} below.  Then by \eqref{ba50} there exists $C_3>0$ and $R_2>0$ such that for $R\geq R_2$
\be
\label{ba6}
\P_\sigma\left(\tau_R^\sigma-\sigma\geq R^{1-\kappa}\right)\leq C_3 R^{1-\kappa-p+p\eps}.
\ee

Now, we can finalize the proof of Theorem \ref{t2}.

\noindent
\smallskip
\textbf{a) Moment estimate \eqref{moment_super}.}
Fix $t_0>0$, take $R> 1\vee t_0^{1/(1-\kappa)}$ sufficiently large.
Fix $t>t_0-R^{1-\kappa}$ and denote
$t_R=t-R^{1-\kappa}>0$.

Applying the estimate \eqref{ba6} with $t_R$ instead of $\sigma$ we get
\be
\P_{t_R}\Big(\tau_R^{t_R}\geq t \Big)
=\P_{t_R}\Big(\tau_R^{t_R} -t_R \geq R^{1-\kappa} \Big)
\leq C_3 R^{1-\kappa-p+p\eps}
\ee
where as usual
\be
\tau_R^{t_R}=\inf\{t\geq t_R\colon |X_t|\leq R \}.
\ee
On the event $\{t\geq \sigma+t_0\}\in \rF_\sigma$ we have $t_R>\sigma$ and therefore
\be
\P_\sigma\Big(\tau_R^{t_R}\geq t \Big)\leq C_3 R^{1-\kappa-p+p\eps}\quad\text{ on }\{t\geq \sigma+t_0\}.
\ee
Recall that the process $V_t=V(X_t)$ with $V$ defined in \eqref{e:V} is a semimartingale whose
drift term $a^V$ is bounded from above by $C_V>0$ due to \eqref{L} in Proposition \ref{p1}. Hence on the event $\{\tau_R^{t_R}<t\}$,
\be
\E_{\tau_R^{t_R}} V_t-V_{\tau_R^{t_R}}\leq \int_{\tau_R^{t_R}}^t C_V\, \di s\leq C_V(t-t_R)=C_V R^{1-\kappa}.
\ee
Since $V_{\tau_R^{t_R}}\leq R^p$, we have finally on the event $\{t\geq \sigma+t_0\}$
\be
\begin{aligned}
\label{tail}
\P_\sigma(|X_t|>2R)&=\P_\sigma(V_t>2^pR^p)\\
&\leq \P_\sigma\Big(\tau_R^{t_R}<t, V_t-V_{\tau_R^{t_R}}>(2^p-1)R^p\Big)+ \P_\sigma(\tau_R^{t_R}\geq t)\\
&\leq (2^p-1)^{-1}R^{-p}\cdot \E_\sigma \Big[ \bI(\tau_R^{t_R}<t)\cdot  (V_t-V_{\tau_R^{t_R}})\Big]+C_3R^{1-\kappa-p+p\eps}\\
&\leq C_4R^{-p+1-\kappa}+C_3R^{1-\kappa-p+p\eps}
\end{aligned}
\ee
for $C_4>0$.
Since $p_X<p+\kappa-1-p\eps$, this yields
\be
\label{moment}
\E_\sigma |X_t|^{p_X}
=p_X\int_0^\infty R^{p_X-1}\P_\sigma(|X_t|>R)\,\di R<\infty \quad  \text{ on the event }\{t\geq \sigma+t_0\}
\ee
and completes the proof of \eqref{moment_super}.

\noindent
\smallskip
\textbf{b) Moment estimate \eqref{PT_super}.}
To prove the passage-time moment bound \eqref{PT_super} we prove by induction the
following extension of \eqref{ba6}: there is $C>0$ such that for and $n\geq 1$ and $R\geq R_1$ large enough
\be
\label{ba7}
\P_\sigma\Big(\tau_R^\sigma-\sigma\geq n R^{1-\kappa}\Big)\leq (C_3 R^{1-\kappa-p+p\eps})^n.
\ee
The induction base is  \eqref{ba6}, which is already proved. To prove the induction step, take $n>1$ and  assume  \eqref{ba7} to be true for $n-1$.
Define
$\sigma_{1}=\sigma+(n-1)R^{1-\kappa}>\sigma$, then using first \eqref{ba6} with $\sigma_1$ instead of
$\sigma$ and then the induction assumption, we get
\be
\begin{aligned}
\P_\sigma\Big(\tau_R^\sigma-\sigma\geq n R^{1-\kappa}\Big)
&=\P_\sigma\Big(\tau_R^\sigma-\sigma\geq (n-1) R^{1-\kappa}, \tau_R^{\sigma_1}-\sigma_1\geq R^{1-\kappa}\Big)\\
&=\E_\sigma\Big[ \bI\Big(\tau_R^\sigma-\sigma\geq (n-1) R^{1-\kappa}\Big)\cdot \P_{\sigma_1}\Big( \tau_R^{\sigma_1}-\sigma_1\geq R^{1-\kappa}\Big)\Big]\\
&\leq C_3R^{1-\kappa-p+p\eps}\cdot
\P_\sigma\Big(\tau_R^\sigma-\sigma\geq (n-1) R^{1-\kappa}\Big)\\
&\leq C_3R^{1-\kappa-p+p\eps}\cdot (C_3 R^{1-\kappa-p+p\eps})^{n-1},
\end{aligned}
\ee
which proves \eqref{ba7} for $n>1$.

Now, for arbitrary $q>0$ we take $n\geq 1$ such that
\be
n(p+\kappa-1-p\eps)>q,
\ee
then by \eqref{ba7} we have for $R\geq R_1$
\be
\P_\sigma(\delta_R^\sigma\geq nR^{1-\kappa})\leq C_3^n R^{-q},
\ee
which yields \eqref{PT_super}. \hfill $\blacksquare$

\subsection{Proof of Theorem \ref{t3}\label{s:t3}} 
We will mainly use the calculations from the proof of Theorem \ref{t2}. The minor modification is that, because we have to obtain the uniform moment estimate for all $t\geq \sigma,$ we have to choose the level $R$ dynamically, i.e. as a function of $t$.

Let $t>0$ be fixed.
On the event $\{t>\sigma\}$ set $R_t=1\vee (t-\sigma)^\frac{1}{1-\kappa}$.
For all $R\geq R_t$ define $t_R\in[\sigma,t)$ by the relation $R=(t-t_R)^\frac{1}{1-\kappa}$.

Repeating literally  the proof of Theorem \ref{t2} we get that for $R>R_t$ the estimate \eqref{tail} holds true:
\be
\label{e:D7}
\begin{aligned}
\P_\sigma(|X_t|>2R)
&\leq C_4 R^{-p} R^{1-\kappa}+C_3R^{1-\kappa-p+p\eps}.
\end{aligned}
\ee
Hence on the event $\{t>\sigma\}\cap\{|X_\sigma|\geq R_t\}$ we get for $\varepsilon>0$ small enough that
\begin{equation}
\begin{aligned}
\E_\sigma |X_t|^{p_X}
&=p_X\int_0^\infty R^{p_X-1}\P_\sigma(|X_t|>R)\,\di R
=p_X\Big(\int_0^{2R_t} + \int_{2R_t}^\infty\Big)R^{p_X-1}\P_\sigma(|X_t|>R)\,\di R\\
&\leq p_X  (2R_t)^{p_X}+ \int_{2}^\infty \Big(C_4 R^{p_X-p-\kappa}+C_3 R^{p_X-p-\kappa+p\varepsilon}\Big)\,\di R\\
&\leq C_6|X_\sigma|^{p_X}+C_5,
\end{aligned}
\end{equation}
where $C_5$ and $C_6>0$ do not depend on $t$ and $\sigma$.

To treat the case $\{|X_\sigma|< R_t\}$ we recall the stopping time $\varsigma_R^\sigma$
of the first jump of $|X|$ larger that $R^{1-\varepsilon}$, and taking into account \eqref{e:NA}
we obtain the
\be
\label{ba51}
\P_\sigma(\varsigma_R^\sigma< t )\leq c_{\nu,p} (t-\sigma) R^{-p+p\eps}.
\ee
Moreover we still have the estimate
\be
\E_{\sigma} V_t-V_{\sigma}\leq \int_{\sigma}^t C_V\, \di s\leq C_V(t-\sigma).
\ee
Hence for on the event $\{\sigma<t\}\cap\{|X_\sigma|<R_t\}$ and for $R\leq R_t$ we get
\be
\begin{aligned}
\P_\sigma(|X_t|>2R)&=\P_\sigma(V_t>2^pR^p)=\P_\sigma\Big(V_t-V_{\sigma}>2^pR^p - V_\sigma^p\Big)\\
&\leq \P_\sigma\Big(V_t-V_{\sigma}>(2^p-1)R^p \Big)\\
&\leq (2^p-1)^{-1}R^{-p}\cdot \E_\sigma \Big[V_t-V_\sigma\Big]\\
&\leq C_7 R^{-p} (t-\sigma).
\end{aligned}
\ee
On the event $\{\sigma<t\}\cap\{|X_\sigma|<R_t\}$ and for $R> R_t$
we introduce the stopping time $\xi^\sigma_{R_t}=\inf\{t\geq \sigma\colon |X_t|\geq R_t\}$. Then by
\eqref{e:D7} with $\xi^\sigma_{R_t}$ instead of $\sigma$ we get
\be
\begin{aligned}
\P_\sigma(|X_t|>2R)&=\P_\sigma(|X_t|>2R,\xi^\sigma_{R_t}\leq t)
=\E_\sigma\Big[ \bI(\xi^\sigma_{R_t}\leq t)\P_{\xi^\sigma_{R_t}}(|X_t|>2R)\Big]     \\
&\leq C_4 R^{-p} R^{1-\kappa}+C_3R^{1-\kappa-p+p\eps}.
\end{aligned}
\ee
Hence on the event $\{\sigma<t\}\cap\{|X_\sigma|<R_t\}$ we finally obtain
\be
\begin{aligned}\label{fin}
\E_\sigma |X_t|^{p_X}
&=p_X\Big[\int_0^{2|X_\sigma|}+\int_{2|X_\sigma|}^{2R_{t}}+\int_{2R_{t}}^\infty\Big] R^{p_X-1}\P_\sigma(|X_t|>R)\,\di R\\
&\leq C_6|X_\sigma|^{p_X}+C_8 (t-\sigma) R_{t}^{p_X-p}+C_5.
\end{aligned}
\ee
Taking into account that $t-\sigma=R_{t}^{1-\kappa}$ we arrive at the uniform estimate
\be
(t-\sigma)R_{t}^{p_X-p}=R_{t}^{1-\kappa+p_X-p}   \leq 1,
\ee
which together with \eqref{fin} yields \eqref{moment_super1}.
 \hfill $\blacksquare$

\appendix

\section{Proofs of the auxiliary estimates\label{a:1}}
\subsection{Proof of \eqref{ba3} in Lemma \ref{l:mart} \label{a:mart1}}
To simplify the notation, let us consider the scalar case $n=1$. In general, one should apply the same estimates component-wise.
We have
\be
\frac{\di}{\di t}
\langle M^{Y,c}\rangle_t\leq \bI_{t\leq \upsilon_R^\sigma} |\nabla U(X_t)|^2B_t\leq c_1R^{-2\kappa}
\ee
for some $c_1>0$.
The process $U_t^+:=\ex^{R^\kappa (M^{Y,c}_{t+ \sigma}-M^{Y,c}_{\sigma})}$ is a submartingale, $U_0^+=1$, and its drift satisfies
\be
a_t^{U^+}\leq c_2 U_t^+,
\ee
for some $c_2>0$.
By the Gronwall lemma, this yields
\be
\E_\sigma  U_{t}^+\leq \ex^{c_2t}, \quad t\geq 0.
\ee
Then by the Doob maximal probability inequality for submartingales, for any $C_1>0$ and $T>0$
\be
\P_\sigma\Big(\max_{\sigma\leq t\leq \sigma+T}(M^{Y,c}_t-M^{Y,c}_\sigma) >C_1 R^{1-\kappa}\Big)
=\P_\sigma\Big(\max_{t\leq T}U_t^+>\ex^{C_1 R}\Big)\leq \ex^{-C_1 R}\E_\sigma U_T^+
\leq  \ex^{-C_1 R}\ex^{c_2T}.
\ee
Repeating the same estimate with the supermartingale $U_t^-:=\ex^{-R^\kappa (M^{Y,c}_{t+ \sigma}-M^{Y,c}_{\sigma})}$ we get a
similar estimate for the
minimum, and hence \eqref{ba3}.

\subsection{Proof of \eqref{ba4} in Lemma \ref{l:mart}\label{a:mart2}}

We follow the same idea as in the previous section  of passing to  exponential submartingales (supermartingales) and using the Doob maximal inequality.
Again, we consider the scalar case $n=1$ only. Take
\be
Q_t^+=\ex^{R^{\kappa-1+\eps}(M^{Y,d}_{t+ \sigma}-M^{Y,d}_{\sigma})},
\ee
then $Q^+$ is a submartingale, $Q_0^+=1$, with the predictable part satisfying for $\sigma\leq t\leq \upsilon_R^\sigma$
\be
\label{e:aQ}
a_t^{Q^+}\leq Q_{t}^+\int_{|z|\leq R^{1-\eps}}\Big(\ex^{R^{\kappa-1+\e}(|X_{t}+z|^{1-\kappa}-|X_{t}|^{1-\kappa})}-1
-R^{\kappa-1+\eps}(|X_{t}+z|^{1-\kappa}-|X_{t-}|^{1-\kappa})\Big)\,K_t(\di z).
\ee
We have for $|x|\geq R$, $|z|\leq R^{1-\eps}$
\be
R^{\kappa-1+\eps}\Big||x+z|^{1-\kappa}-|x|^{1-\kappa}\Big|\leq c_1R^{-1+\eps}|z|\leq c_1
\ee
for some $c_1>0$
and thus applying the estimate $\ex^a-1-a\leq c_2 a^2$, $|a|\leq c_1$  we get
\be
\label{e:expR}
\Big(\ex^{R^{\kappa-1+\eps}(|x+z|^{1-\kappa}-|x|^{1-\kappa})}-1-R^{\kappa-1+\e}(|x+z|^{1-\kappa}-|x|^{1-\kappa})\Big)\leq c_2R^{-2+2\eps}|z|^2.
\ee
If $p\geq 2$, we have simply
\be
\label{e:p2}
\int_{|z|\leq R^{1-\eps}}|z|^2\,K_t(\di z)\leq c_{\langle N\rangle}.
\ee
Otherwise,
\be
\label{e:p0}
\int_{|z|\leq R^{1-\eps}}|z|^2\,K_t(\di z)\leq c_{\langle N\rangle}+\int_{1<|z|\leq R^{1-\eps}}|z|^p(R^{1-\eps})^{2-p} \,K_t(\di z)
\leq c_{\langle N\rangle}+(R^{1-\eps})^{2-p}c_{\nu,p}\leq c_3 R^{(1-\e)(2-p)}
\ee
for some $c_3>0$.
That is, with the help of \eqref{e:aQ}, \eqref{e:expR} and \eqref{e:p2} we have for $p\geq 2$
\be
a_t^{Q^+}\leq c_2\cdot c_{\langle N\rangle}\cdot Q_{t}^+  R^{-2(1-\e)}
\ee
and with the help of \eqref{e:aQ}, \eqref{e:expR} and \eqref{e:p0} for $0<p<2$
\be
\begin{aligned}
a_t^{Q^+}\leq c_2\cdot c_3\cdot Q_{t}^+R^{-2+2\eps}R^{(2-p)(1-\eps)}=c_2\cdot c_3\cdot Q_{t}^+ R^{-p(1-\e)}
\end{aligned}
\ee
otherwise. Since $\eps\in(0,1)$, this yields in any case
\be
a_t^{Q^+}\leq c_4 Q_t^+.
\ee
Applying the Gronwall lemma, we get $\E_\sigma Q_t\leq \ex^{c_4 t}$. The rest of the proof is almost the same as in the previous section, namely
\be
\P_\sigma\Big(\max_{\sigma\leq t\leq \sigma+T}(M^{Y,d}_t-M^{Y,d}_\sigma) >C_1 R^{1-\kappa}\Big)
=\P_\sigma\Big(\max_{t\leq T}Q_t^+ >\ex^{C_1 R^\eps}\Big)\leq \ex^{-C_1 R^\eps}\E_\sigma Q_T^+
\leq  \ex^{-C_1 R^\eps}\ex^{c_4 T}.
\ee
Using the same argument for the supermatringale
\be
Q_t^-=\ex^{-R^{\kappa-1+\eps}(M^{Y,d}_{t+ \sigma}-M^{Y,d}_{\sigma})}
\ee
instead of $Q_t^-$, we get the estimate from below
and hence \eqref{ba4}.

\subsection{Proof of \eqref{ba5}\label{a:varsigma}}

Denote  $A=\{z\colon |z|>R^{1-\eps}\}$, then by \textbf{A}$_{\nu,p}$ for $R>1$
\be
\nu(A\times [s,t])\leq R^{-p(1-\eps)}
\int_s^t\int_{|z|>1} |z|^{p}\, K_r(\di z)\, \di r\leq c_{\nu,p}R^{-p(1-\eps)}  (t-s)\quad \hbox{ a.s. for all $s<t$,}
\ee
Recall that $N_A(t):=N(A\times[0,t])$ is
a counting process and $\nu(A\times [0,t])$ is its compensator. Hence for any stopping time $\sigma$ and $h>0$ we have
\be
\label{e:NA}
\begin{aligned}
\P_{\sigma}(N_A(\sigma+h)-N_A(\sigma)>0)&=\P_{\sigma}(N_A(\sigma+h)-N_A(\sigma)\geq 1)
\\&\leq \E_{\sigma}\Big[N_A(\sigma+h)-N_A(\sigma)\Big]=\E_{\sigma}\nu(A\times [\sigma, \sigma+h])\leq c_{\nu,p}R^{-p(1-\eps)} h.
\end{aligned}
\ee
Applying this inequality with $h= R^{1-\kappa}$ we get
\be
\P_\sigma (\varsigma_R^\sigma-\sigma\leq R^{1-\kappa} )
=\P_\sigma(N(A\times[\sigma,\sigma+R^{1-\kappa}])>0)
\leq c_{\nu,p} R^{1-\kappa-p(1-\eps)},
\ee
which proves \eqref{ba5}.

\section{Infinite moments \label{a:2}}

Let $\kappa\in [-1,1)$.
We show the divergence of $q_X$-moments, $q_X\geq \alpha+\kappa-1$, of the process $X$ defined in \eqref{e:storage} in Section \ref{s:opt}.

Let $t>0$. Take $R>1$ and assume $Z$ to have a jump of the value $z>2R$ at time moment
$\tau\leq t$. Then $X_t$ is bounded from below by the solution to ODE
\be
\di x_t=-r(x_t)\, \di t, \quad t\geq \tau, \quad x_\tau=z.
\ee
The solution to this ODE is given explicitly, at least up to the passage-time of level $x=1$ by $x_t$, namely
\be
x_t=\Big(z^{1-\kappa}-(1-\kappa)(t-\tau)\Big)^\frac{1}{1-\kappa},\quad 0\leq t-\tau\leq \frac{z^{1-\kappa}-1}{1-\kappa},\quad z\geq 1.
\ee
Thus, for $t-\tau\leq \frac{2^{1-\kappa}-1}{1-\kappa}R^{1-\kappa}$, we have
\be
x_t\geq \Big((2R)^{1-\kappa}-(1-\kappa)\frac{2^{1-\kappa}-1}{1-\kappa}R^{1-\kappa}\Big)^\frac{1}{1-\kappa}=R>1.
\ee
This bound yields the following. Denote $c_\kappa:=\frac{2^{1-\kappa}-1}{1-\kappa}$. Then for $R>1$
\be
\begin{aligned}
\P(X_t>R)
&\geq \P\Big(\text{there exists $\tau\geq 0$ such that $\Delta Z_\tau>2R$ and $t-\tau\leq c_\kappa R^{1-\kappa}$}\Big)\\
&=\P\Big( N\big([0\vee (t-c_\kappa R^{1-\kappa}),t]\times (2R,\infty) \big)   \geq 1 \Big)\\
&=1-\exp\Big( - \big(t- (t-c_\kappa R^{1-\kappa})\vee 0\big)\cdot (2R)^{-\alpha}   \Big)   \\
&=1-\exp\Big(-(c_\kappa R^{1-\kappa})\wedge t) \cdot (2R)^{-\alpha}\Big)=:f_t(R),
\end{aligned}
\ee
and therefore
\be
\E X_t^{q_X}=q_X\int_{0}^{\infty} R^{q_X-1}\P(X_t>R)\, \di R\geq q_X\int_{1}^{\infty} R^{q_X-1} f_t(R)\, \di R.
\ee
We have
\be
f_t(R)\nearrow f_\infty(R)=1-\exp\Big(-2^{-\alpha}c_\kappa R^{1-\alpha-\kappa}\Big),\quad t\to\infty,
\ee
and therefore  by the monotone convergence theorem
\be
\liminf_{t\to\infty}\E X_t^{q_X}\geq q_X\int_{1}^{\infty} R^{q_X-1}f_\infty(R)\, \di R.
\ee
If $\alpha+\kappa\leq 1$, i.e.\ the balance condition \eqref{e:bc} fails, then for $R>1$
\be
f_\infty(R)\geq c_0,\quad c_0=1-\ex^{-2^{-\alpha} c_\kappa}>0.
\ee
Hence since $q_X> 0$ we get
\be
q_X\int_{1}^{\infty} R^{q_X-1}f_\infty(R)\, \di R\geq c_0q_X \int_{1}^{\infty} R^{q_X-1}\, \di R=+\infty.
\ee
If $\alpha+\kappa> 1$, i.e.\ the balance condition \eqref{e:bc} holds, then for $R>1$
\be
f_\infty(R)\geq c_{\alpha,\kappa}R^{1-\kappa-\alpha}, \quad c_{\alpha,\kappa}=2^{-\alpha}c_\kappa\inf_{x\in (0, 2^{-\alpha}c_\kappa]}\frac{1-\ex^{-x}}{x}>0.
\ee
Hence since $q_X\geq \alpha+\kappa-1$ we get
\be
q_X\int_{1}^{\infty} R^{q_X-1}f_\infty(R)\, \di R \geq c_{\alpha,\kappa} q_X\int_1^\infty R^{q_X-\alpha-\kappa}\, \di R=+\infty.
\ee

\section{Proof of Proposition~\ref{p1} \label{apC}}
\begin{proof}
We will use the ``time localization+Fatou lemma'' trick, similar to the argument used in Section \ref{s41}.
 Let $\tau_n\nearrow\infty$ be a localizing sequence for the local martingale part $M^V$ in the
semimartingale decomposition \eqref{DM}. We can and will assume that $C_V\geq 0$. Since $V\geq 0$, by Fatou's lemma and \eqref{L}
we have on the event $\{t> \sigma\}$
\be
\E_\sigma V_t
\leq \liminf_{n\to \infty} \E_\sigma V_{t\wedge\tau_n}
\leq \liminf_{n\to \infty}
\Big(   V_{\sigma\wedge \tau_n}+\E_\sigma\int_{\sigma\wedge \tau_n}^{t\wedge\tau_n} \Big(C_V-c_V V_s^\gamma\Big)\, \di s\Big).\\
\ee
We have $V_{\sigma\wedge \tau_n}\to V_\sigma, n\to \infty$, and by the monotone convergence theorem
\be
\E_\sigma \int_{\sigma\wedge \tau_n}^{t\wedge\tau_n} V_s^\gamma\, \di s\to \E_\sigma \int_{\sigma}^{t}V_s^\gamma \, \di s.
\ee
This gives the bound
\be
\label{C1}
\E_\sigma V_t\leq V_\sigma+ \E_\sigma \int_{\sigma}^{t}(C_V-c_VV_s^\gamma) \, \di s \quad \text {on } \{t> \sigma\},
\ee
which lead to all the statements claimed. Namely, simply neglecting the term $-c_VV_s^\gamma$ we get
\be
\E_\sigma V_t\leq V_\sigma+ C_V(t-\sigma) \quad \text {on } \{t> \sigma\},
\ee
which proves \emph{(i)}. Next, we rewrite \eqref{C1} as
\be
 c_V \E_\sigma \int_{\sigma}^{t}V_s^\gamma \, \di s\leq   V_{\sigma}
+ C_V\E_\sigma (t-\sigma)  \quad \text {on } \{t> \sigma\},
\ee
and dividing the both sides of this inequality by $c_V(t-\sigma)$ we get \emph{(ii)}. 
Finally, by the Jensen inequality, for $\gamma\geq 1$ \eqref{C1} yields
\be
\label{C2}
\E_\sigma V_t\leq V_\sigma+ \int_{\sigma}^{t}(C_V-c_V\E_\sigma  V_s^\gamma) \, \di s   \quad \text {on } \{t> \sigma\}.
\ee
Changing in this inequality $\sigma, t$ to $t, t+\epsilon$ with $\eps>0$ and taking a conditional expectation w.r.t.\ to 
$\rF_{\sigma}$, we get
\be
\E_\sigma V_{t+\eps}-\E_\sigma V_t\leq \int_{t}^{t+\eps}(C_V-c_V\E_\sigma  V_s^\gamma) \, \di s  \quad \text {on } \{t> \sigma\}.
\ee
Dividing by $\eps$ and passing to the limit $\eps \searrow 0$, we get the following inequality for the right derivative of $\E_\sigma V_t$:
\be
\frac{\di^+}{\di t}\E_\sigma V_t\leq C_V-c_V\E_\sigma  V_t^\gamma  \quad \text {on } \{t> \sigma\}.
\ee
Combined with the initial condition $\E_\sigma V_t|_{t=\sigma}=V_\sigma$, this yields \emph{(iii)}.

\end{proof}


\end{document}